\documentclass[english]{amsart}
\usepackage{dsfont}
\usepackage{url}
\usepackage[utf8]{inputenc}
\usepackage[T1]{fontenc}
\usepackage{lmodern}
\usepackage{babel}
\usepackage{mathtools}  
\usepackage{amssymb}
\usepackage{lipsum}
\usepackage{mathrsfs}
\usepackage{color}

\newtheorem{theorem}{Theorem}[section]
\newtheorem{proposition}{Proposition}[section]
\newtheorem{lemma}{Lemma}[section]
\newtheorem{definition}{Definition}[section]
\newtheorem{corollary}{Corollary}[section]

\theoremstyle{remark}
\newtheorem{remark}{Remark}[section]

\numberwithin{equation}{section}

\newfont{\sBlackboard}{msbm10 scaled 900}

\newcommand{\mylabel}[1]{\label{#1}
            \ifx\undefined\stillediting
            \else \fbox{$#1$}\fi }
\newcommand{\BE}{\begin{equation}}

\newcommand{\EEQ}{\end{equation}}
\newcommand{\rfb}[1]{\mbox{\rm
   (\ref{#1})}\ifx\undefined\stillediting\else:\fbox{$#1$}\fi}

\newfont{\Blackboard}{msbm10 scaled 1200}

\newfont{\roma}{cmr10 scaled 1200}

\def\CC{\rm \hbox{C\kern-.56em\raise.4ex
         \hbox{$\scriptscriptstyle |$}\kern+0.5 em }}
\newcommand{\be}{\begin{equation}}
\newcommand{\ee}{\end{equation}}
\newcommand{\beq}{\begin{eqnarray}}
\newcommand{\eeq}{\end{eqnarray}}
\newcommand{\beqs}{\begin{eqnarray*}}
\newcommand{\eeqs}{\end{eqnarray*}}
\newcommand{\bt}{\begin{Theorem}}
\newcommand{\et}{\end{Theorem}}
\newcommand{\br}{\begin{remark}}
\newcommand{\er}{\end{remark}}
\newcommand{\bc}{\begin{Corollary}}
\newcommand{\ec}{\end{Corollary}}
\newcommand{\el}{\end{Lemma}}
\newcommand{\bd}{\begin{definition}}
\newcommand{\ed}{\end{definition}}

\newcommand{\mm}    {{\hbox{\hskip 0.5pt}}}

\newcommand{\bluff} {{\hbox{\raise 15pt \hbox{\mm}}}}
\makeatletter
\def\section{\@startsection {section}{1}{\z@}{-3.5ex plus -1ex minus
    -.2ex}{2.3ex plus .2ex}{\large\bf}}
\makeatother
\def\be{\begin{equation}}
\def\ee{\end{equation}}

\def\ds{\displaystyle}
\usepackage{mathrsfs}

\newcommand{\mathcurl}{\mathscr}
\newcommand{\Con}{\ensuremath{\mathcurl C}}

\newcommand{\Cinf}{\ensuremath{\Con^\infty}}

\newcommand{\imp}{\Rightarrow}

\begin{document}

\thispagestyle{empty}
\title[Magnetic Schr\"odinger equations]{Comments on observability and stabilization of magnetic Schr\"odinger equations}


\date\today

\author[Ka\"{\i}s Ammari]{Ka\"{\i}s AMMARI}
\address{UR Analysis and Control of Pde, UR 13ES64, Department of Mathematics, Faculty of Sciences of Monastir, University of Monastir, 5019 Monastir, Tunisia and LMV-UVSQ/Universit\'e de Paris-Saclay, France} \email{kais.ammari@fsm.rnu.tn} 
 
\author[Mourad Choulli]{Mourad Choulli}
\address{Universit\'e de Lorraine, 34 cours L\'eopold, 54052 Nancy cedex, France}
\email{mourad.choulli@univ-lorraine.fr}
 
\author[Luc Robbiano]{Luc Robbiano}
\address{Laboratoire de Math\'ematiques, Universit\'e de Versailles Saint-Quentin en Yvelines, 78035 Versailles, France}
\email{luc.robbiano@uvsq.fr}

\begin{abstract}
We are mainly interested in extending the known results on observability inequalities and stabilization for the Schr\"odinger equation to the magnetic Schr\"odinger equation. That is in presence of a magnetic potential. We establish observability inequalities and exponential stabilization by extending the usual multiplier method, under the same geometric condition  that needed for the Schr\"odinger equation. We also prove, with the help of elliptic Carleman inequalities, logarithmic stabilization results through a resolvent estimate. Although the approach is classical, these results on logarithmic stabilization seem to be new even for the Schr\"odinger equation.
\end{abstract}

\maketitle

\tableofcontents

\section{Introduction}
\setcounter{equation}{0}

Prior to give the precise statement of our main results we need to consider IBVP's for magnetic Sch\"odinger equation that we are interested in.
 For this, we firstly give the main notations and the preliminary results that we will use frequently in this text
 
\subsection{Notations and preliminaries}

Denote by $dx$ the Lebesgue measure on $\mathbb{R}^d$, $d\ge 1$, and $d\sigma$ the Lebesgue measure on a submanifold $S$ of $\mathbb{R}^k$ of dimension $k-1$. 
Let $X$ be an open subset of $\mathbb{R}^d$ and $Y=(X,d\mu)$, $Y=(S,d\mu)$ or $Y=(X\times S, d\mu)$, where $d\mu = dx$ if $Y=X$, $d\mu =d\sigma $ if $Y=S$ and $d\mu =dx\otimes d\sigma$ if $Y=X\times S$.

\smallskip
For $f,g\in L^2(Y)=L^2(Y,\mathbb{C})$ and $E\subset Y$ is measurable, we set
\begin{align*}
&(f|g)_{0,E}=\int_Ef\overline{g}d\mu ,
\\
&\|f\|_{0,E}=\left(\int_E|f|^2d\mu\right)^{1/2}
\end{align*}
and, if in addition $Y=(X,d\mu)$ and $f\in H^1(Y)=H^1(Y,\mathbb{C})$, let
\[
\|f\|_{1,E}=\left(\|f\|_{0,E}^2+\sum_{j=1}^d\|\partial_jf\|_{0,E}^2\right)^{1/2}.
\]

Similarly, for $F,G\in L^2(Y,\mathbb{C}^\ell )$, $\ell \ge 1$, we define
\begin{align*}
&(F|G)_{0,E}=\int_EF\cdot \overline{G}d\mu ,
\\
&\|F\|_{0,E}=\left(\int_E|F|^2d\mu\right)^{1/2}.
\end{align*}


Finally, for $f\in L^\infty (X,\mathbb{R}^\ell)$, $\ell \ge 1$, we set
\[
\|f\|_\infty =\||f|\|_{L^\infty (X,\mathbb{R})}.
\]

Throughout this text, $\Omega$ is a $C^\infty$ bounded domain of $\mathbb{R}^n$, $n \ge 1$, with  boundary $\Gamma$. Let $\nu$ denotes the outward unit normal vector field on $\Gamma$. 

Henceforth $\mathbf{a} = (a_1,\ldots ,a_n) \in W^{1,\infty}(\Omega,\mathbb{R}^n)$ is a fixed vector field. We define the magnetic Laplacian and  the magnetic gradient respectively by
\[
\Delta_\mathbf{a}= \ds \sum_{j=1}^n (\partial_j + i \, a_j)^2=\Delta +2i\mathbf{a}\cdot \nabla +i\mbox{div}(\mathbf{a})-|\mathbf{a}|^2
\]
and
\[
\nabla_\mathbf{a} = \nabla + i \mathbf{a}. 
\]

We shall also need  the notation
\[
\partial _{\nu_\mathbf{a}} = \nabla_\mathbf{a} \cdot \nu=\partial_\nu +i\mathbf{a}\cdot \nu.
\]

The following identities will be useful in the sequel. There are obtained by making integrations by parts
\begin{align}
&(\Delta_\mathbf{a} f|g)_{0,\Omega}=-(\nabla_\mathbf{a}f|\nabla_\mathbf{a} g)_{0,\Omega}+(\partial_{\nu_\mathbf{a}}f|g)_{0,\Gamma},\quad f\in H^2(\Omega ),\; g\in H^1(\Omega) ,\label{i1}
\\
&(\Delta_\mathbf{a} f|g)_{0,\Omega}=(f|\Delta_\mathbf{a}g)_{0,\Omega},\quad f,g\in H^2(\Omega )\cap H_0^1(\Omega ).\label{i2}
\end{align}

Note that we can take $f\in H_\Delta (\Omega )$ in \eqref{i1} instead of $f\in H^2(\Omega )$. In that case $(\partial_{\nu_\mathbf{a}}f|g)_{0,\Gamma}$ has to be interpreted as a duality pairing between $\partial_{\nu_\mathbf{a}}f\in H^{-1/2}(\Gamma )$ and $g\in H^{1/2}(\Gamma )$.

Let $\Lambda$ be a nonempty open subset of $\Gamma$ and 
\begin{equation}\label{hs}
\mathcal{H}=\{u\in H^1(\Omega);\; u=0\; \mbox{on}\; \Lambda\}.
\end{equation}
The Poincar\'e constant of $\mathcal{H}$ will be denoted by $\varkappa (\mathcal{H})$. That is $\varkappa (\mathcal{H})$ is the best constant so that
\[
\|u\|_{0,\Omega}\le C\|\nabla u\|_{0,\Omega},\;\; u\in \mathcal{H}.
\]
We have in particular
\[
\|u\|_{0,\Omega}\le \varkappa (\mathcal{H})\|\nabla u\|_{0,\Omega},\;\; u\in \mathcal{H}.
\]

\textbf {Magnetic gradient semi-norm}. Consider on $H_0^1(\Omega )=H_0^1(\Omega ,\mathbb{C})$ the semi-norm
\[
f\in H_0^1(\Omega )\mapsto \|\nabla_\mathbf{a} f\|_{0,\Omega}.
\] 
As it is shown by Esteban and Lions \cite[page 406]{EL}, we have 
\[
|\nabla |f||\le |\nabla_\mathbf{a}f|\quad \mbox{a.e. in}\; \Omega .
\]
Indeed, bearing in mind that $\mathbf{a}$ takes its values in $\mathbb{R}^n$, we have
\[
|\nabla |f||=\left|\Re \left(\nabla f \frac{\overline{f}}{|f|}\right)\right|= \left|\Re \left((\nabla f +i\mathbf{a}f)\frac{\overline{f}}{|f|}\right)\right|\quad \mbox{a.e. in}\; \Omega .
\]

As a consequence of this relation, we deduce that $\|\nabla_\mathbf{a} \cdot \|_{0,\Omega}$ defines a norm on $H_0^1(\Omega )$. This norm is not in general equivalent to the natural norm $\|\nabla \cdot \|_{0,\Omega}$ on $H_0^1(\Omega)$. For simplicity sake's, even it is not always necessary, we assume that $\mathbf{a}$ is chosen is such a way that $\|\nabla_\mathbf{a} \cdot \|_{0,\Omega}$ is equivalent to $\|\nabla \cdot \|_{0,\Omega}$. This is achieved for instance if $0$ is not an eigenvalue of the $\Delta_\mathbf{a}$, under Dirichlet boundary condition. We refer to \cite[Proposition 3.1]{BC} for a proof and other equivalent conditions. 

Note that if $\|\mathbf{a}\|_\infty $ is sufficiently small then $\|\nabla_\mathbf{a} \cdot \|_{0,\Omega}$ and  $\|\nabla \cdot \|_{0,\Omega}$ are equivalent on $H_0^1(\Omega)$. This follows in a straightforward manner by observing that if $\varkappa =\varkappa \left(H_0^1(\Omega)\right)$, then
\begin{equation}\label{pi}
(1-\|\mathbf{a}\|_\infty\varkappa )\|\nabla u\|_{0,\Omega}\le \|\nabla _\mathbf{a}u\|_{0,\Omega}\le (1+\|\mathbf{a}\|_\infty\varkappa )\|\nabla u\|_{0,\Omega}.
\end{equation}
Whence, under the smallness condition
\begin{equation}\label{sc}
\|\mathbf{a}\|_\infty <\frac{1}{\varkappa},
\end{equation}
$\|\nabla_\mathbf{a} \cdot \|_{0,\Omega}$ and  $\|\nabla \cdot \|_{0,\Omega}$ are equivalent on $H_0^1(\Omega )$.

More generally, if $\mathcal{H}$ is of the form \eqref{hs} and $\|\mathbf{a}\|_\infty <\frac{1}{\varkappa (\mathcal{H})}$, then $\|\nabla_\mathbf{a} \cdot \|_{0,\Omega}$ and  $\|\nabla \cdot \|_{0,\Omega}$ are equivalent on $\mathcal{H}$.

\subsection{IBVP's for the magnetic Schr\"odinger operator}

Consider $\Gamma_0$ and $\Gamma_1$ two disjoint nonempty open subsets of $\Gamma$ so that $\Gamma=\overline{\Gamma_0}\cup\overline{\Gamma_1}$. 

We consider henceforward the following assumptions on the damping coefficients.

$(A_c)$ $0\le c\in L^\infty (\Omega )$ and there exist $\omega$, an open subset of $\Omega$, and $c_0>0$ so that $c\ge c_0$  a.e. in $\omega$.

$(A_d)$ $0\le d\in L^\infty (\Gamma _0)$ and there exist $\gamma_0$, an open subset of  $\Gamma_0$, and $d_0>0$ so that $d\ge d_0$  a.e. on $\gamma_0$.

We deal with systems governed  by IBVP's for the magnetic Schr\"odinger operator with different types of dampings. The first system we consider is given by the IBVP
\be \label{schint1}
\left\{
\begin{array}{ll}
i u_t + \Delta_{\mathbf{a}} u + i c(x) u= 0 \;\; &\mbox{in}\; \Omega \times (0,+\infty), 
\\
u = 0  &\mbox{on}\; \Gamma \times (0,+\infty),
\\
 u(\cdot ,0) = u_0.
 \end{array}
 \right.
\ee

As a  consequence of \eqref{i2}  we obtain that the unbounded operator $A:L^2(\Omega )\rightarrow L^2(\Omega )$ given by $Au=\Delta_\mathbf{a}u$ and $D(A)= H^2(\Omega )\cap H_0^1(\Omega )$ is self-adjoint. 

From inequality \eqref{i1}
\[
\Re ( Au|u) _0=-\|\nabla_\mathbf{a} u\|_0^2\le 0,\quad u\in D(A).
\]
As a non negative self-adjoint densely defined operator, $A$ is m-dissipative. Then so is $A_0= iA, D(A_0) = D(A),$ and, consequently, $A_0$ generates a strongly continuous group $e^{tA_0}$. 

Assume that $c$ obeys to assumption $(A_c)$ and let $A_1=i\Delta_\mathbf{a}-c$ with domain $D(A_1)=D(A_0)$. As a bounded perturbation of $A_0$, $A_1$ generates a strongly continuous semigroup $e^{tA_1}$.

For $u_0\in L^2(\Omega )$, define the energy for the system \eqref{schint1} by
\[
\mathcal{E}^1_{u_0}(t)=\frac{1}{2}\left\|e^{tA_1}u_0\right\|_{0,\Omega}^2.
\]
If $u(t)= e^{tA_1}u_0$, we get by using identity \eqref{i1}
\[
\frac{d}{dt}\|u(t)\|_{0,\Omega}^2=2\Re (u'(t),u(t))_{0,\Omega}=2\Re \left[i\| \nabla_\mathbf{a}u(t)\|_{0,\Omega}^2-\|\sqrt{c}u(t)\|^2_{0,\Omega}\right],\;\; t>0.
\]
Hence
\[
\frac{d}{dt}\mathcal{E}^1_{u_0}(t)=-\|\sqrt{c}u(t)\|^2_{0,\Omega},\;\; t>0.
\]

Therefore $t \mapsto \mathcal{E}^1_{u_0}(t)$ is decreasing when $u_0\ne 0$. We can then address the question to know how fast this energy decay. This issue will be one of our objectives in the coming sections.

The second system is associated with an IBVP with boundary damping.
\be \label{schfront1}
\left\{
\begin{array}{ll}
i u_t + \Delta_{\mathbf{a}} u = 0 \;\; &\mbox{in}\; \Omega \times (0,+\infty), 
\\
\partial_{\nu_\mathbf{a}}u +du_t= 0  &\mbox{on}\; \Gamma_0 \times (0,+\infty),
\\
u = 0  &\mbox{on}\; \Gamma_1 \times (0,+\infty),
\\
 u(\cdot ,0) = u_0.
 \end{array}
 \right.
\ee

Introduce
\[
V=\{ u\in H^1(\Omega );\; u_{|\Gamma_1}=0\}.
\]

Then, as we have seen before,  under the smallness condition
\begin{equation}\label{sc1}
\|\mathbf{a}\|_\infty <\frac{1}{\varkappa (V)},
\end{equation}
$\|\nabla_\mathbf{a} \cdot \|_{0,\Omega}$ and  $\|\nabla \cdot \|_{0,\Omega}$ are equivalent on $V$. In particular, $V$ endowed with the norm $\|\nabla_\mathbf{a}\cdot \|_{0,\Omega}$ is a Hilbert space.

Let $d$ satisfies assumption $(A_d)$ and consider  the unbounded operator $A_2:V\rightarrow V$ given by
\[
A_2=i\Delta _\mathbf{a}\;\; \mbox{and}\;\; D(A_2)=\{ u\in V;\;  \Delta _\mathbf{a} u\in V\; \mbox{and}\; \partial_{\nu_\mathbf{a}}u+id\Delta_\mathbf{a}u=0\; \mbox{on}\; \Gamma _0\}.
\]
Let $W=\{ u\in V;\; \Delta_a u \in V \; \mbox{and}\; \partial_{\nu_\mathbf{a}}u\in L^2(\Gamma _0)\}$. 
Apply then twice \eqref{i1} in order to derive, for $u,v\in W$,
\begin{align}
&(\nabla_\mathbf{a}(i\Delta_\mathbf{a}u)|\nabla_\mathbf{a}v)_{0,\Omega}=-i( \Delta_\mathbf{a}u|\Delta_\mathbf{a}v)_{0,\Omega}+i(\Delta_\mathbf{a}u|\partial_{\nu_\mathbf{a}}v)_{0,\Gamma_0}\label{disp1},
\\
& (\nabla_\mathbf{a}u|\nabla_\mathbf{a}(-i\Delta_\mathbf{a}v))_{0,\Omega}=-i( \Delta_\mathbf{a}u|\Delta_\mathbf{a}v)_{0,\Omega}+i(\partial_{\nu_\mathbf{a}}u|\Delta_\mathbf{a}v)_{0,\Gamma_0}\label{disp2}.
\end{align}
Take in \eqref{disp1} and \eqref{disp2} $u\in D(A_2)$ and $v\in W$, we find
\begin{equation}\label{disp3}
(\nabla_\mathbf{a}(i\Delta_\mathbf{a}u)|\nabla_\mathbf{a}v)_{0,\Omega}=(\nabla_\mathbf{a}u|\nabla_\mathbf{a}(-i\Delta_\mathbf{a}v))_{0,\Omega}+i(\Delta_\mathbf{a}u|\partial_{\nu_\mathbf{a}}v-id\Delta_av)_{0,\Gamma_0}.
\end{equation}
Pick $\varphi \in C^\infty(\overline{\Omega}\setminus\Gamma_1)$ and let $u\in V$ be the variational solution of the BVP
\[
\left\{
\begin{array}{ll}
-\Delta_a u=\varphi \quad &\mbox{in}\; \Omega ,
\\
\partial_{\nu_\mathbf{a}}u=id\varphi &\mbox{in}\; \Gamma_0.
\end{array}
\right.
\]
It is then not hard to check that $u\in D(A_2)$. Hence $(\Delta_\mathbf{a}u|\partial_{\nu_\mathbf{a}}v-id\Delta_av)_{0,\Gamma_0}=0$ for any $u\in D(A_2)$ implies in particular $(\varphi |\partial_{\nu_\mathbf{a}}v-id\Delta_av)_{0,\Gamma_0}=0$ for any $\varphi \in C^\infty(\overline{\Omega}\setminus\Gamma_1)$. Hence $\partial_{\nu_\mathbf{a}}v-id\Delta_av=0$ on $\Gamma_0$. From this and \eqref{disp3} we obtain that 
\[
A_2^\ast=-i\Delta _\mathbf{a}\;\; \mbox{and}\;\; D(A_2)=\{ u\in V;\;  \Delta _\mathbf{a} u \in V\; \mbox{and}\; \partial_{\nu_\mathbf{a}}u-id\Delta_\mathbf{a}u=0\; \mbox{on}\; \Gamma _0\}.
\]
Here we identified the Hilbert space $V$ with its dual space.

Now $u=v$, with $u\in D(A_2)$, in \eqref{disp1} yields
\begin{equation}\label{disp4}
\Re (\nabla_\mathbf{a}(A_2u)|\nabla_\mathbf{a}u)_{0,\Omega}=-\|\sqrt{d}\Delta_{\mathbf{a}}u\|_{0,\Gamma_0}^2  \le 0.
\end{equation}
We get similarly from \eqref{disp2}, where $u\in D(A_2^\ast)$,
\[
\Re (\nabla_\mathbf{a}(A_2^\ast u)|\nabla_\mathbf{a}u)_{0,\Omega}=-\|\sqrt{d}\Delta_{\mathbf{a}}u\|_{0,\Gamma_0}^2  \le 0.
\]
In other words, $A_2$ and $A_2^\ast$ are dissipative. On the other hand, as for the Laplace operator, one can prove that $A_2$ is closed graph. Therefore, according to \cite[Proposition 3.1.11, page 73]{tucsnakbook}, $A_2$ is m-dissipative. Whence $A_2$ is the generator of strongly continuous semigroup $e^{tA_2}$.

The energy associated to the system \eqref{schfront1} is given by
\[
\mathcal{E}_{u_0}^2(t)=\frac{1}{2}\left\| \nabla_\mathbf{a} e^{tA_2}u_0\right\|_{0,\Omega}^2,\;\; u_0\in V.
\]
In light of \eqref{disp4}, we have
\[
\frac{d}{dt}\mathcal{E}_{u_0}^2(t)=\Re ( \nabla_\mathbf{a}u(t)|\nabla_\mathbf{a}u'(t)) _{0,\Omega}=\Re(\nabla_\mathbf{a}(A_2u(t))|\nabla_\mathbf{a}u(t))_{0,\Omega}=-\|\sqrt{d}\Delta_{\mathbf{a}}u(t)\|_{0,\Gamma_0},\quad t>0.
\]
Here again, we see that $t \mapsto \mathcal{E}_{u_0}^2(t)$ is decreasing whenever $u_0\ne 0$.

The third system is again an IBVP with a boundary damping
\be \label{schfront1b}
\left\{
\begin{array}{ll}
i u_t + \Delta_{\mathbf{a}} u = 0 \;\; &\mbox{in}\; \Omega \times (0,+\infty), 
\\
\partial_{\nu_\mathbf{a}}u -idu= 0  &\mbox{on}\; \Gamma_0 \times (0,+\infty),
\\
u = 0  &\mbox{on}\; \Gamma_1 \times (0,+\infty),
\\
 u(\cdot ,0) = u_0.
 \end{array}
 \right.
\ee
Define the unbounded operator $A_3:L^2(\Omega )\rightarrow L^2(\Omega )$ given by
\[
A_3=i\Delta _\mathbf{a}\;\; \mbox{and}\;\; D(A_3)=\{ u\in V;\;  \Delta _\mathbf{a} u \in L^2(\Omega)\; \mbox{and}\; \partial_{\nu_\mathbf{a}}u-idu=0\; \mbox{on}\; \Gamma _0\},
\]
where $d$ obeys to assumption $(A_d)$.

We repeat the same argument that we used to prove that $A_2$ is dissipative in order to derive that $A_3$ is also $m$-dissipative. Therefore $A_3$ generates a strongly continuous semigroup $e^{tA_3}$.

The energy corresponding to the system \eqref{schfront1b} is
\[
\mathcal{E}^3_{u_0}(t)=\frac{1}{2}\left\|e^{tA_3}u_0\right\|_{0,\Omega}^2,\;\; u_0\in L^2(\Omega ).
\]
In light of identity \eqref{i1}, for $u,v\in D(A_3)$, we have
\[
(i\Delta_\mathbf{a}u|v)_{0,\Omega}=-i(\nabla _\mathbf{a}u|\nabla_\mathbf{a}v)_{0,\Omega}-(du|v)_{0,\Gamma_0}.
\]
Whence
\[
\frac{d}{dt}\mathcal{E}^3_{u_0}(t)=-\|\sqrt{d}u(t)\|_{0,\Gamma_0}^2 ,\quad t>0,
\]
where $u(t)=e^{tA_3}u_0$, $u_0\in L^2(\Omega)$.

One more time, we observe that, if $u_0\neq 0$ then $t \mapsto \mathcal{E}^3_{u_0}(t)$ is decreasing.

If $\overline{\Gamma_0}\cap \overline{\Gamma_1}\not=\emptyset$, we do not have necessarily $D(A_j) \subset H^2(\Omega )$. In order to avoid this case, we assume in the rest of this text, even if it is not always necessary, that  $\overline{\Gamma_0}\cap \overline{\Gamma_1}=\emptyset$. In other words, $\Gamma$ has at least two connected components.

Prior to give sufficient condition guaranteeing that $D(A_j) \subset H^2(\Omega )$, $j=2,3$, we introduce, for $s\in \mathbb{R}$ and $1\leq r\leq \infty$, 
\[
B_{s,r}(\mathbb{R}^{n-1} ):=\{ w\in \mathscr{S}'(\mathbb{R}^{n-1});\; (1+|\xi |^2)^{s/2}\widehat{w}\in 
L^r(\mathbb{R}^{n-1} )\},
\]
where $\mathscr{S}'(\mathbb{R}^{n-1})$ is the space of temperated distributions on $\mathbb{R}^{n-1}$ and
$\widehat{w}$ is the Fourier transform of $w$. Endowed with its natural norm 
\[
\|w\|_{B_{s,r}(\mathbb{R}^{n-1} )}:=\|(1+|\xi |^2)^{s/2}\widehat{w}\|_{ L^r(\mathbb{R}^{n-1} )},
\]
$B_{s,r}(\mathbb{R}^{n-1} )$ is a Banach space (it is noted that $B_{s,2}(\mathbb{R}^{n-1} )$ is merely the usual Sobolev
space $H^s(\mathbb{R}^{n-1} )$). By using local charts and a partition of unity, we construct 
$B_{s,r}(\Gamma )$ from $B_{s,r}(\mathbb{R}^{n-1})$ similarly as $H^s(\Gamma )$ is built 
from $H^s(\mathbb{R}^{n-1})$.

The main interest in these spaces is that the multiplication by a function from $B_{s,1}(\Gamma _0)$, $s\ge 0$, defines a bounded operator on $H^s(\Gamma _0)$ (see \cite[Theorem 2.1]{Ch2002}).

Additionally to the previous conditions on $\mathbf{a}$ and $d$, we assume in the rest of this text that $\mathbf{a}\cdot \nu \in B_{1/2,1}(\Gamma _0)$ and  $d\in B_{1/2,1}(\Gamma _0)$.

Under these supplementary assumptions, for $u\in D(A_j)$, $j=2,3$, $\partial_\nu u\in H^{1/2}(\Gamma _1)$ and, since \[\left[2i\mathbf{a}\cdot \nabla +i\mbox{div}(\mathbf{a})-|\mathbf{a}|^2\right]u\in L^2(\Omega ),\] the usual $H^2$-regularity for the Laplacian with mixed boundary conditions entail $u\in H^2(\Omega )$. Whence, $D(A_j)\subset H^2(\Omega )$, $j=1,2,3$. 

\begin{remark}
{\rm
1) Let $\psi \in W^{2,\infty}(\Omega ,\mathbb{R})$ and denote by $A_j^\psi$, $j=1,2,3$, the operator $A_j$ in which we substituted $\mathbf{a}$ by $\mathbf{a}+\nabla \psi$. Straightforward computations give 
\[
e^{-i\psi}\nabla_\mathbf{a}e^{i\psi}=\nabla_{\mathbf{a}+\nabla \psi},\quad e^{-i\psi}\Delta_\mathbf{a}e^{i\psi}=\Delta_{\mathbf{a}+\nabla \psi}
\]
 and then
\begin{equation}\label{r1}
e^{tA_j^\psi}=e^{-i\psi}e^{tA_j}e^{i\psi},\;\; j=1,2,3.
\end{equation}
In particular,
\[
\|e^{tA_j^\psi}\|_{\mathscr{B}(H)}=\|e^{tA_j}\|_{\mathscr{B}(H)},\;\; H=L^2(\Omega )\; \mbox{for}\; j=1,3\; \mbox{and}\; H=V\; \mbox{if}\; j=2.
\]
Let $\mathcal{E}^{j,\psi}_{u_0}$ the energy corresponding to $A_j^\psi$, with $u_0\in L^2(\Omega )$, $j=1,3$ and $u_0\in V$ for $j=2$. In light of \eqref{r1}, we have
\[
\mathcal{E}^{j,\psi}_{u_0}=\mathcal{E}^j_{e^{i\psi}u_0},\;\; j=1,2,3.
\]

2) Assume $n=1$ and let $\Omega =(0,1)$. Denote by $A_j^0$ the operator $A_j$ when $\mathbf{a}=0$, $j=1,2,3$. Using that $\psi (x) =\int_0^x\mathbf{a}(t)dt$ satisfies $\partial_x\psi =\mathbf{a}$, we get from 1)
\[
e^{tA_j}=e^{-i\psi}e^{tA_j^0}e^{i\psi}\;\; \mbox{and}\;\; \mathcal{E}^{j}_{u_0}=\mathcal{E}^{0,j}_{e^{i\psi}u_0},\;\; j=1,2,3.
\]
Here $\mathcal{E}^{0,j}_{u_0}$ is the energy corresponding to $A_j^0$, $j=1,2,3$. Therefore, all the results existing in the literature without the presence of magnetic potential can be transferred to the magnetic case.

}
\end{remark}

\subsection{Main results}

Let $\mathcal{H}_\ell =L^2(\Omega )$ if $\ell=1,3$ and $\mathcal{H}_2=V$. The first results we are going to prove concern logarithmic stabilization in both cases of interior or boundary damping. We will prove in each case of $\ell =1$, $\ell=2$ or $\ell =3$,

\begin{theorem} 
	\label{maintheorem1}
Assume that assumptions $(A_c)$ and $(A_d)$ hold. For every $\mu\in\mathbb{R}$, $A_\ell -i\mu$ is invertible and

\noindent
{\rm (i)} $\|(A_\ell -i\mu)^{-1}\|_{\mathscr{B}(\mathcal{H}_\ell)}\le C e^{K\sqrt{|\mu|}}$, $\mu\in\mathbb{R}$, for some  constants $C>0$ and $K>0$,

\noindent
{\rm (ii)} there exists a constant $C_1>0$ such that 
\[
\|e^{tA_\ell }u_0\|_{\mathcal{H}_\ell}\le \frac{C_1}{\ln^{2k}(2+t)}\| u_0\|_{D(A_\ell ^k)},\quad  u_0\in D(A_\ell ^k) .
\]
\end{theorem}

Next, we establish observability inequalities for the magnetic Schr\"odinger operator. To this end, fix $x_0\in \mathbb{R}^n$, let $m=m(x)=x-x_0$, $x\in \mathbb{R}^n$ and assume that
\begin{equation}\label{ass1}
\Gamma_0=\{ x\in \Gamma ;\; m(x)\cdot \nu (x)>0\}.
\end{equation}

Observe that in the present case the condition $\overline{\Gamma_0}\cap\overline{\Gamma_1}=\emptyset$ is satisfied for instance if $\Omega =\Omega_0\setminus \Omega_1$, with $\Omega_1\Subset \Omega_0$, $\Omega_j$ star-shaped with respect to $x_0\in \Omega_1$ and $\Gamma_j=\partial \Omega_j$, $j=0,1$.

\begin{proposition}\label{mainproposition1}
Assume that $\Gamma_0$ is of the form \eqref{ass1}. Then there exists a constant $C>0$,  only depending on $\Omega$ and $T$, so that,  for any $u_0\in D(A_0)$ and $u(t)=e^{tA_0}u_0$, we have
\[
\|\nabla_\mathbf{a}u_0\|_{0,\Omega}  \le C\|\partial_{\nu_\mathbf{a}}u\|_{0,\Sigma_0}.
\]
\end{proposition}

Consider the following assumption: (A) $\omega$ is a neighborhood of $\Gamma_0$ in $\Omega$ so that $\overline{\omega}\cap \Gamma_1=\emptyset$.


\begin{proposition}\label{mainproposition2}
Under assumption (A), there exists a constant $C>0$,  only depending on $\Omega$, $T$, $\Omega$ and $\Gamma_0$, so that, for any $u_0\in D(A_0)$ and $u(t)=e^{tA_0}u_0$, we have
\[
\|u_0\|_{0,\Omega}  \le C\|u\|_{0,Q_\omega}.
\]
Here $Q_\omega =\omega \times (0,T)$.
\end{proposition}

Finally, we use these observability inequalities to obtain the following exponential stabilization results.

\begin{theorem}\label{maintheorem2}
If the assumption (A) holds, then there exists a constant $\varrho>0$, depending only on $\Omega$, $T$, $\Omega$ and $\Gamma_0$, so that
\[
\mathcal{E}_{u_0}^1(t)\le e^{-\varrho t}\mathcal{E}_{u_0}^1(0),\quad  u_0\in L^2(\Omega ).
\]
\end{theorem}

\begin{theorem}\label{maintheorem3}
Let $\Gamma_0$ be of the form \eqref{ass1} for some $x_0$. Then there exists $0<\varsigma \le \frac{1}{2\varkappa (V)}$, depending on $x_0$ and $\Omega$, with the property that, if $\|\mathbf{a}\|_\infty \le \varsigma$ and $\mathbf{a}=0$ on $\Gamma_0$, then there exists two constants $C>0$ and $\varrho>0$, depending only on $x_0$ and $\Omega$, so that
\[
\mathcal{E}_{u_0}^2(t)\le Ce^{-\varrho t}\mathcal{E}_{u_0}^2(0),\quad u_0\in V.
\]
\end{theorem}

\subsection{State of art}
Observability inequalities for the Schr\"odinger equation were established by Machtyngier \cite{mach} by the multiplier method. The corresponding exponential stabilization results are due to Machtyngier and Zuazua \cite{zuazua}.  Our observability inequalities together with exponential stabilisation extend those in \cite{mach,zuazua}.

Under the so-called geometric control condition, Lebeau \cite{lebeau} showed that the Schr\"odinger equation is exactly controlable (or equivalently exactly observable) for an arbitrary fixed time (see also Phung \cite{phung}, Laurent \cite{laurent} and Dehman, G\'erard and Lebeau \cite{DGL} for the nonlinear case). In the case of a square, Ramdani, Takahashi, Tenenbaum and Tucsnak \cite{RTTT} obtained an observability inequality by a spectral method which is build on the fact that observability is equivalent to an observality resolvent estimate, known also as Hautus test. This equivalence was first proved by Burq and Zworski \cite{burq2} (see also Miller \cite{miller}).

Early observability estimates for the Schr\"odinger equation on torus were established by Haraux \cite{Haraux} and Jaffard \cite{Jaffard} in two dimensions and without potentials. The case of Schr\"odinger equation on spheres and Zoll manifolds was studied in Maci\`a \cite{macia}, Marci\`a and Rivi\`ere \cite{MR1,MR2}. The observability inequalities for the Schr\"odinger equation  on the torus and the disk was also considered by Anantharaman, Fermanian-Kammerer and Maci\`a \cite{AFM},  Anantharaman and M. L\'eautaud \cite{AL},   Anantharaman, M. L\'eautaud and Maci\`a \cite{ALM}, Anantharaman and Maci\`a \cite{AM}. Its is worth mentioning  that the results in \cite{AFM, ALM, AM} allow time-dependent potentials and these results hold without any geometric condition on the observation set.

Exact observability inequalities for the (magnetic) wave equation can transferred to observability inequalities for the (magnetic) Schr\"odinger equation and vice versa via a transmutation method (see Miller \cite{miller} and references therein) or by an abstract framework consisting in transforming a second order evolution equation into a first order evolution equation (see \cite[Theorem 6.7.5 and Proposition 6.8.2]{tucsnakbook} for more details). 

There is wide literature on control, observability and stabilization for the wave equation. We only quote the following few reference \cite{blr, burq3, fursikov, gr, zk, Ro}.

\subsection{Outline}

The rest of this text is organized as follows. Section 2 is devoted to establish logarithmic decay of each of the energies $\mathcal{E}_{u_0}^j$, $j=1,2,3$. The main step consists in proving a resolvent estimate via elliptic Carleman inequalities. Logarithmic energy decay is obtained by using an abstract theorem guaranteeing such decay when the resolvent satisfies some estimates.  We note that the logarithmic stability results we establish in Section 2 hold without any geometric condition. We revisit in Section 3 the multiplier method with the objective to extend the existing results for the Sch\"ordinger equation to the magnetic Sch\"ordinger equation, provided that the magnetic potential satisfies certain conditions. In Section 3, we need the usual geometric conditions on the control subregion. Namely, the boundary control region must contain a  part of the boundary enlightened by a point in the space. For the internal control region, its boundary must contain again a part of the boundary enlightened by a point in the space. In the last section, we added supplementary comments. Precisely, we give an exponential stabilization estimate based on a direct application of a Carleman inequality and an observability inequality in a product space.

\section{Logarithmic stabilization} \label{carleman}

We firstly recall some interior Carleman estimates as well as  boundary Carleman estimates. 
For this last case we have several estimates depending 
on the a priori knowledge we have on traces. Next, we apply these inequalities in order to get resolvent estimates on imaginary axis, yielded to obtain
energy decay of logarithmic type.

\subsection{Carleman estimates}

Carleman estimates can be viewed as weighted energy estimates with a large parameter. The crucial assumption is the sub-ellipticity condition introduced in this context by
H\"ormander \cite{Hormander}. 

Henceforth 
\[
X=(-2,2)\times \Omega\;\; \mbox{and}\;\; L=(-2,2)\times \Gamma.
\]

Let $P$ be equal to the Laplace operator plus an  operator of order 1 with bounded coefficients. The principal symbol of $P$ is then $p(y,\eta)=|\eta|^2$. 

Set, for $\varphi \in \Cinf(\mathbb{R}^{n+1},\mathbb{R})$,  
\[
p_\varphi(y,\eta,\tau)=p(y,\eta +i\tau \nabla\varphi(y)).
\]

\begin{definition}  
 \label{def: sub-ellipticity}
 Let $\mathcal{O}$ be a bounded open set in $\mathbb{R}^{n+1}$ and $\varphi \in \Cinf(\mathbb{R}^{n+1},\mathbb{R})$. We say that $\varphi $ satisfies the  {\em
   sub-ellipticity} condition in $\overline{\mathcal{O}}$ if $|\nabla \varphi|>0$ in $\overline{\mathcal{O}}$
 and 
  \begin{equation}
    \label{eq: sub-ellipticity}
    p_\varphi(y,\eta,\tau)=0,\; (y,\eta) \in \overline{\mathcal{O}} \times\mathbb{R}^{n+1},\; \tau>0\;
   \imp        \;  \{\Im p_\varphi ,\Re p_\varphi \}(y,\eta ,\tau) >0,
 \end{equation}
 where $\{\cdot ,\cdot\}$ is the usual Poisson bracket.
\end{definition}

\begin{remark}
{\rm
	\label{rem: construction sub-ellipticity function}
Note that the sub-ellipticity condition is not really too restrictive. To see that, pick $\psi \in \Cinf(\mathbb{R}^{n+1},\mathbb{R})$ such that $\nabla \psi(y)\ne 0$ for every $y\in 
\overline{\mathcal{O}}$. Then $\varphi(y)=e^{\lambda\psi(y)}$ satisfies obviously the sub-ellipticity property  in $\overline{\mathcal{O}}$  if $\lambda$ is chosen 
sufficiently large. This gives a method to construct a weight  function having  the  
   sub-ellipticity property in $\overline{\mathcal{O}}$ but other choices could be possible.
   }
\end{remark}


\subsubsection{Interior Carleman estimate}

The following Carleman estimate is classical and we can find a proof in H\"ormander~\cite[Theorem 8.3.1]{Hormander}.
\begin{theorem}\label{t2.1}
Let $U$ be an open subset of $X$ and assume that $\varphi $ obeys to the sub-ellipticity condition in $\overline{U}$. 
Then there exist
$C>0$ and $\tau_0>0$, such that
\begin{equation}\label{ice}
\tau^3\| e^{\tau\varphi} f \| _{0,X} ^2+\tau\| e^{\tau\varphi} \nabla f \| _{0,X}^2\le C\| e^{\tau\varphi}  Pf \|_{0,X}^2,
\end{equation}
for all $\tau\ge \tau_0$, $f\in\Con_0^\infty(U)$.
\end{theorem}

\subsubsection{Boundary Carleman estimates}
For simplicity sake's, we use in the sequel the notation
\[
Y=(-2,2)\times \overline{\Omega}.
\]
Let $y_0\in L$ and $\mathcal{O}$ be a neighborhood of $y_0$ in $(-2,2)\times \mathbb{R}^{n} $.
We  say that  $f\in \overline{ \Con_0^\infty}(\mathcal{O}_{|X})    $ if  there exists $g\in \Con_0^\infty(\mathcal{O})$ such that $f=g_{|Y}$. In particular 
$f\in\Con^\infty(Y)$. 
This definition allows functions with non null traces on $\partial X$ but with null traces on 
$\partial ( \mathcal{O}\cap X) \setminus L$. The following theorem is proved in~\cite[Proposition 1]{Lebeau-Rob97}.

\begin{theorem}\label{t2.2}
Let $y_0\in\partial X$ and $\mathcal{O}$ a neighborhood of $y_0$ in $(-2,2)\times \mathbb{R}^n$ and assume that $\varphi $ satisfies
 the sub-ellipticity condition in $\overline{\mathcal{O}\cap X}$. We also assume that $\partial_\nu \varphi(y)\ne 0$  in $\overline{\partial \mathcal{O}\cap\partial X}$. Then there exist
$C>0$ and $\tau_0>0$, such that
\begin{multline*}
\hskip 1cm \tau^3\| e^{\tau\varphi} f \|_{0,X} ^2
+\tau\| e^{\tau\varphi} \nabla f \| _{0,X}^2
+\tau\| e^{\tau\varphi}\nabla f\|^2_{0,L}
\\
 \le C\left(\| e^{\tau\varphi}  Pf \|_{0,X} ^2
+\tau^3\| e^{\tau\varphi}f\|^2_{0,L}
+\tau\| e^{\tau\varphi}\partial_\nu f\|^2_{0,L}\right),\hskip 2cm
\end{multline*}
for all $\tau\ge \tau_0$, $f\in \overline{ \Con_0^\infty}(\mathcal{O}_{| X})  $.
\end{theorem}

This Carleman estimate is useful when we know Dirichlet and Neumann traces of $f$ on a part of the boundary. It allows to estimate the function $f$  in an interior domain 
by its Dirichlet and Neumann traces on a part of the boundary and $Pf$. 

The two next theorems only assume that the knowledge of the Dirichlet trace or Neumann trace. They 
permit to estimate 
the function $f$ up to the boundary by $Pf$ and a priori knowledge of $f$ in a small domain contained in $X$.

Henceforth, $\nabla_T$ denotes the tangential gradient on $\Sigma$. The following theorem is proved in~\cite[Proposition 1]{Lebeau-Rob95}
\begin{theorem}\label{t2.3}
Let $y_0\in\partial X$ and $\mathcal{O}$ a neighborhood of $y_0$ in $(-2,2)\times \mathbb{R}^n$, assume that $\varphi $ satisfies
 the sub-ellipticity condition in $\overline{\mathcal{O}\cap X}$ and $\partial_\nu \varphi(y)< 0$  on
 $\overline{\partial \mathcal{O}\cap\partial X}$. Then there exist
$C>0$ and $\tau_0>0$, such that
\begin{multline*}
\hskip 1cm \tau^3\| e^{\tau\varphi} f \| _{0,X} ^2
+\tau\| e^{\tau\varphi} \nabla f \| _{0,X} ^2
+\tau\| e^{\tau\varphi}\partial_\nu f\|^2_{0,L} \\
  \le C\left(\| e^{\tau\varphi}  Pf \|_{0,X} ^2
+\tau^3\| e^{\tau\varphi}f\|^2_{0,L }
+\tau\| e^{\tau\varphi}\nabla _Tf\|^2_{0,L}\right),\hskip 2cm
\end{multline*}
for all $\tau\ge \tau_0$, $f\in \overline{ \Con_0^\infty}(\mathcal{O}_{| X})  $.
\end{theorem}

The following theorem is a consequence of~\cite[Lemma 4]{Lebeau-Rob97}.
\begin{theorem}\label{t2.4}
Let $y_0\in\partial X$ and $\mathcal{O}$ a neighborhood of $y_0$ in $(-2,2)\times \mathbb{R}^n$, assume that $\varphi $ satisfies
 the sub-ellipticity condition in $\overline{\mathcal{O}\cap X}$ and $\partial_\nu \varphi(y)< 0$  on
 $\overline{\partial \mathcal{O}\cap\partial X}$. Then there exist
$C>0$ and $\tau_0>0$, such that
\begin{multline*}
\hskip 1cm\tau^3\| e^{\tau\varphi} f \|_{0,X}  ^2
+\tau\| e^{\tau\varphi} \nabla f \|_{0,X}  ^2
+\tau^3\| e^{\tau\varphi}f\|^2_{0,L}
+\tau\| e^{\tau\varphi}\nabla f\|^2_{0,L}
\\
 \le C\left(\| e^{\tau\varphi}  Pf \| _{0,X} ^2
+\tau\| e^{\tau\varphi}\partial_\nu f\|^2_{0,L}\right), \hskip 2cm
\end{multline*}
for all $\tau\ge \tau_0$, $f\in \overline{ \Con_0^\infty}(\mathcal{O}_{|X })  $.
\end{theorem}

\subsubsection{Global Carleman estimates}

We can patch together the interior and boundary Carleman estimates to obtain a global one. The global Carleman estimate we obtain will be very useful to tackle the stabilization issue for system \eqref{schint1}.

\begin{theorem}
	\label{th: Carleman global interior Dirichlet}
Let $Z$ be a open subset of $X$ and assume that $\varphi $  satisfies
 the sub-ellipticity condition in  $Y\setminus Z$. Assume moreover that $\partial_\nu \varphi(y)< 0$  in $L$. Then there exist
$C>0$ and $\tau_0>0$, such that
\begin{align*}
&C\left( \tau^3\| e^{\tau\varphi} f \|_{0,X}  ^2
+\tau\| e^{\tau\varphi} \nabla f \|_{0,X}  ^2
+\tau\| e^{\tau\varphi}\partial_\nu f\|^2_{0,L} \right)
\\
& \hskip 2cm \le \| e^{\tau\varphi}  Pf \|_{0,X} ^2
+\tau^3\| e^{\tau\varphi}f\|^2_{0,L}
+\tau\| e^{\tau\varphi}\nabla _Tf\|^2_{0,L}
\\
&\hskip 7cm + \| e^{\tau\varphi} f \| ^2_{0,Z}+\| e^{\tau\varphi} \nabla f \| ^2_{0,Z},
\end{align*}
for all $\tau\ge \tau_0$, $f\in { \Con^\infty}(\overline{X})  $.
\end{theorem}
 
We now state  a theorem that we will  use to deal with stabilization issue for systems \eqref{schfront1} and \eqref{schfront1b} 

Set  
\[
L_j=(-2,2)\times \Gamma_j,\;\; j=0,1.
\]
\begin{theorem}
	\label{th: Global Carleman estimate boundary Neumann Dirichlet}
Let $\Lambda$ an open subset of $L_0$. Assume that $\varphi$ satisfies
 the sub-ellipticity condition in  $Y$ and $\partial_\nu \varphi(y)< 0$  in
 $L \setminus \Lambda$.
  Then there exist
$C>0$ and $\tau_0>0$, such that
\begin{align*}
&C\left(\tau^3\| e^{\tau\varphi} f \| _{0,X} ^2
+\tau\| e^{\tau\varphi} \nabla f \|_{0,X}  ^2
+\tau^3 \| e^{\tau\varphi} f\|^2_{0,L } 
+\tau \| e^{\tau\varphi}\nabla f\|^2_{0,L} \right)
\\
&\hskip 2cm \le \| e^{\tau\varphi}  Pf \| _{0,X} ^2
+\tau^3\| e^{\tau\varphi}f|^2_{0,L_1}
+\tau\| e^{\tau\varphi}\nabla _Tf\|^2_{0,L_1 }
\\
&\hskip 4cm +\tau\| e^{\tau\varphi}\partial_\nu f\|^2_{0,L_0 \setminus \Lambda} 
+\tau^3\| e^{\tau\varphi}f\|^2_{0,\Lambda}
+\tau\| e^{\tau\varphi}\partial_\nu f\|^2_{0,\Lambda},
\end{align*}
for all $\tau\ge \tau_0$, $f\in { \Con^\infty}(\overline{X})  $.
\end{theorem}

To prove Theorems \ref{th: Carleman global interior Dirichlet} and \ref{th: Global Carleman estimate boundary Neumann Dirichlet} we can  proceed similarly to the proof of \cite[Lemma 8.3.1]{Hormander}. A rough idea of the proof is the following. Assume that we have a Carleman estimate in a neighborhood $\mathcal{U}$ of each point of $\overline{X}$. If $(U_j)$ is a finite covering of $\overline{X}$ of such neighborhoods, we pick $(\chi_j)$ a partition of unity subordinate to this covering. In each $U_j$, we apply the corresponding Carleman estimate to $\chi_j u$, i.e. Theorem \ref{t2.1} if $U_j\subset X$ or  one of Theorem \ref{t2.2}, \ref{t2.3} and  \ref{t2.4} if $U_j\cap \partial X \ne \emptyset$, depending on assumptions we have on boundary terms. Putting together all these estimates in order to get, in a classical way,  Theorems \ref{th: Carleman global interior Dirichlet} and \ref{th: Global Carleman estimate boundary Neumann Dirichlet}.

\begin{remark}
{\rm
	\label{rem: construction weight function}
All the previous theorems still hold if we substitute $P$ by $P$ plus a first order operator $Q$ having bounded coefficients. For that it is enough to observe that 
\[
\| e^{\tau\varphi}  Pf \|_{0,X}\le \| e^{\tau\varphi}  (P+Q)f \|_{0,X}+\| e^{\tau\varphi}  Qf \|_{0,X}
\]
and that the term $\| e^{\tau\varphi}  Qf \|_{0,X}$  can be absorbed by the left hand side of \eqref{ice}, by modifying $\tau_0$ if necessary. 
}
\end{remark}

The assumptions on the weight function may impose some constraints on the topology of $\Omega$. In Theorem~\ref{th: Carleman global interior Dirichlet}, if $\varphi$ satisfies 
$\partial_\nu \varphi(y)< 0$  in $L$, $\varphi$ has a maximum in $X$, thus we have to impose that  this maximum belongs to $Z$. 
 In Theorem~\ref{th: Global Carleman estimate boundary Neumann Dirichlet}, we need $\nabla\varphi\ne0$ in $Y$. This is always possible as long as we do not assume that $\partial_\nu\varphi$ is of constant sign on $Z$. However one can  construct weight functions
 $\varphi$ obeying to the assumptions of the preceding theorems.
 
 \begin{proposition}
 	\label{prop: psi function interior}
 Let $Z$ an open subset of $X$. 
There exists $\psi\in\Con^\infty(\overline{X})$ such that $\partial_\nu\psi<0$ on $L$ and $\nabla\psi\ne0$ in 
$Y\setminus Z$.
 \end{proposition}
 
 \begin{proposition}
  	\label{prop: psi function boundary}
 Let $\Lambda$ be an open subset of $L_0$. 
There exists $\psi\in\Con^\infty(\overline{X})$ such that $\partial_\nu\psi<0$ on
$L \setminus \Lambda$ and $\nabla\psi\ne0$ in $Y$.
 \end{proposition}
 
To prove the existence of such functions $\psi$, we first construct $\psi$ in a neighborhood of $L$ 
(resp.  $L\setminus \Lambda$). Next, we extend this function to $Y$ and approximate the extended function  by a Morse function. Finally, we {\em push} the singularities in $Z$ along paths to singularities in a point in $Z$ (resp.   in the exterior of $X$ along paths passing through  $\Lambda$).
We refer for instance \cite[Section 14.2, page 437]{tucsnakbook} for a proof. One can then check that $\varphi=e^{\lambda\psi} $ possesses the assumptions of 
Theorem \ref{th: Carleman global interior Dirichlet} (resp. Theorem \ref{th: Global Carleman estimate boundary Neumann Dirichlet})
 for $\psi$ constructed in Proposition \ref{prop: psi function interior} (resp. Proposition \ref{prop: psi function boundary}).

\subsection{Stabilization by a resolvent estimate}

The resolvent set of an operator $B$ will denoted by $\rho(B)$.

The following abstract theorem is the key tool in establishing the logarithmic stabilization for each of the three systems we are interested in.

\begin{theorem}
	\label{th: resolvent and log decay}
Let $B$ the generator of a continuous semigroup $e^{tB}$ on a Hilbert space $H$. Assume that

\noindent
{\rm (i)} $ \ds \sup_{t\ge0} \| e^{tB}\|_{\mathscr{B}(H)}<\infty$,

\noindent
{\rm (ii)} $i\mathbb{R}\subset \rho(B)$,

\noindent
{\rm (iii)}  $ \| (B-i\mu)^{-1}\|_{\mathscr{B}(H)}\le C e^{K\sqrt{|\mu|}}$, $\mu\in\mathbb{R}$, for some constants $C>0$ and $K>0$.

\noindent
Then  there exists a constant $C_1>0$, such that
 \[ 
 \|e^{tB}f\|_{H}\le \frac{C_1}{\ln ^{2k}(2+t)}\| f\|_{D(B^k)},\;\; f\in D(B^k)
 \]
 or equivalently 
\[
\|e^{tB}B^{-k}\|_{\mathscr{B}(H)}\le \frac{C_1}{\ln ^{2k}(2+t)}  . 
\]
\end{theorem}
This result is a particular case of \cite[Theorem 1.5]{batty}.
 
\subsubsection{Interior damping}

We deal in this subsection with the system \eqref{schint1}. Specifically we are going to apply Theorem \ref{th: resolvent and log decay} with $B=A_1$ and $H=L^2(\Omega)$. That is we will prove Theorem \ref{maintheorem1} when $\ell=1$. We restate here for convenience this result.

\begin{theorem} 
	\label{th: pb1 carleman approach}
Assume that assumption $(A_c)$ is satisfied. For every $\mu\in\mathbb{R}$, $A_1-i\mu$ is invertible and

\noindent
{\rm (i)} $\|(A_1-i\mu)^{-1}\|_{\mathscr{B}(L^2(\Omega ))}\le C e^{K\sqrt{|\mu|}}$, $\mu\in\mathbb{R}$, for some  constants $C>0$ and $K>0$,

\noindent
{\rm (ii)} there exists a constant $C_1>0$, such that 
\[
\|e^{tA_1}u_0\|_{L^2(\Omega )}\le \frac{C_1}{\ln^{2k}(2+t)}\| u_0\|_{D(A_1^k)},\;\; u_0\in D(A_1^k) .
\]
\end{theorem}

\begin{proof}
Let us first consider the resolvent equation $(A_1-i\mu)u=g$, $g\in L^2(\Omega)$. Changing $g$ by $-ig$, we are lead to solve 
\begin{equation}
	\label{eq: pb1 resolvent}
\Delta_\mathbf{a}u+icu-\mu u=g. 
\end{equation}
Multiplying this equation by $\overline u$ and  integrating on $\Omega$, we have 
\begin{equation*}
(\Delta_\mathbf{a}u| u)_{0,\Omega}+i(cu|u)_{0,\Omega}-\mu(u|u)_0=(g|u)_{0,\Omega},
\end{equation*}
We obtain by applying \eqref{i1}
\begin{equation}
\label{eq: pb1 inner product}
-\|\nabla_\mathbf{a}u\|_{0,\Omega}^2+i(cu|u)_{0,\Omega}-\mu\| u\|_{0,\Omega}^2=(g|u)_{0,\Omega}.
\end{equation}

Taking the real part of this equation, we obtain
\begin{equation*}
-\|\nabla_\mathbf{a}u\|_{0,\Omega}^2-\mu\| u\|_{0,\Omega}^2=\Re( g|u)_{0,\Omega} .
\end{equation*}
If $\mu\ge 0$, this estimate entails
\[
\|\nabla_\mathbf{a}u\|_{0,\Omega}^2\le \| g\|_{0,\Omega}\|u\|_{0,\Omega}
\]
and hence
\[
\| u\|_{0,\Omega}\le \varkappa ^2k^2\| g\|_{0,\Omega},\;\; \mu\ge0.
\]
Here $\varkappa$ is the Poincar\'e constant of $H_0^1(\Omega)$ and $k$ is a constant so that $\|\nabla w\|_{0,\Omega}\le k\|\nabla_\mathbf{a}w\|_{0,\Omega}$, for each $w\in H_0^1(\Omega)$. In other words, we proved the resolvent estimate when $\mu \ge 0$.

Next,  simple computations show that $(iA_1)^\ast= iA_1+2ic$. Whence
\[
\mbox{ind}(iA_1+\mu)=-\mbox{ind}(iA_1+2ic+\mu)=-\mbox{ind}(iA_1+\mu)
\]
and then $\mbox{ind}(iA_1+\mu )=0$. Therefore,  $A_1-i\mu$ is invertible if and only if  it is injective. 

To prove that $A_1-i\mu$ is injective, take, for $g=0$, the imaginary part of equation~\eqref{eq: pb1 inner product} in order to obtain that $u=0$ in $\omega$. Hence 
$\Delta_\mathbf{a}u+icu-\mu u=0$ in $\Omega$ and $u=0$ in $\omega$. Then, by the unique continuation property, $u=0$ in $\Omega$. 

We complete the proof by establishing the resolvent estimate when $\mu<0$. By continuity argument, we are reduced to prove the resolvent estimate for large $|\mu|$. To do that, we obtain, by taking again the imaginary part of 
equation~\eqref{eq: pb1 inner product}, 
\begin{equation}
	\label{eq: pb1 estimate on omega}
c_0 \|u\|_{0,\omega}^2\le \| u\|_{0,\Omega} \|g\|_{0,\Omega}.
\end{equation}

Now are now going to apply a Carleman inequality to estimate $\|u\|_{0,\Omega}$ in terms of $\|u\|_{0,\omega}$. To this end, define $f(s,x)=e^{\alpha s}u(x)$, where $\alpha=\sqrt{-\mu }$. Since $u$ is the solution of \eqref{eq: pb1 resolvent}, we easily get that  $f$ satisfies
\begin{equation}
	\label{eq: pb1 equation adding variable}
\partial_s^2 f+\Delta_\mathbf{a} f+ icf= e^{s\alpha }g.
\end{equation}
Fix  $\omega' \Subset\omega$ and set 
\[
X_1=(-1,1)\times \Omega \;  \mbox{and}\;\; X_2= (-1/2,1/2)\times \omega'.
\]
Pick $\chi\in\Con_0^\infty(\mathbb{R})$, such that $\chi(s)=1$ for $|s|\le 3/4$ and $\chi(s)=0 $ for $ |s|\ge1$. We put
\[
\varphi(s,x)=e^{\lambda(-\beta s^2+\psi(x))},
\] 
where $\psi$ satisfies Proposition~\ref{prop: psi function interior} with $Z=X_2$ and $\beta>0$ is fixed in what follows.
 The critical points of $-\beta s^2+\psi(x)$ are located in
$X_2$. Then for $\lambda $ sufficiently large (but fixed from now on) $\varphi$ satisfies the sub-ellipticity condition according 
to Remark~\ref{rem: construction sub-ellipticity function}. We can apply Theorem~\ref{th: Carleman global interior Dirichlet}, 
with  $\chi f$ instead of $f$. We obtain as $\chi f$ satisfies the Dirichlet boundary condition
\begin{align}
	\label{eq: pb 1 carleman consequence}
\tau^3\| e^{\tau\varphi} \chi f \| ^2_{0,X}
+\tau\| e^{\tau\varphi} \nabla ( \chi f )\| ^2_{0,X}
&\lesssim \| e^{\tau\varphi}  \big(  \partial_s^2(\chi f)+\Delta_\mathbf{a}(\chi f)+ ic\chi f  \big) \|^2_{0,X}
 \\
&\hskip 2.5cm + \| e^{\tau\varphi} f \| ^2_{0,X_2}
+\| e^{\tau\varphi} \nabla f \| ^2_{0,X_2}.\notag
\end{align}
Here and until the end of this proof, $Q_1\lesssim Q_2$ means that $Q_1\le CQ_2$, for some generic constant $C$, only depending on $\Omega$, $\psi$, $\mathbf{a}$ and $c$.

We have 
\[ 
\partial_s^2(\chi f)+\Delta_\mathbf{a}(\chi f)+ ia\chi f  = e^{s\alpha } \chi g+ 2 \partial_s\chi  \partial_sf+f \partial_s^2\chi . 
\]
As $\partial_s \chi$ is supported in the set  $\{ s\in\mathbb{R}, 3/4\le |s|\le 1\}$, we get
\begin{equation}\label{E1}
\|  e^{\tau\varphi} (2 \partial_s\chi  \partial_sf+f \partial_s^2\chi)\|^2_{0,X}
\lesssim \alpha e^{C_1\tau+2\alpha  }\| u\|^2_{0,\Omega},
\end{equation}
with $C_1=e^{\lambda ( -9\beta/16+\max_\Omega\psi)} $.

On the other hand 
\begin{align}
&\| e^{\tau\varphi} f \| ^2_{0,X_2}+\| e^{\tau\varphi} \nabla f \| ^2_{0,X_2}\lesssim \alpha
 e^{\tau C_3 +2\alpha }\| u\|^2_{1,\omega'}, \label{E2}
 \\
 &\| e^{\tau\varphi}e^{s\alpha } \chi g \|_{0,X}^2\lesssim     e^{\tau C_3 +2\alpha }   \| g\|^2_{0,X}.\label{E3}
 \end{align}
where $C_3=2 e^{\lambda \max_\Omega\psi}$.
 
Inequalities \eqref{E1}, \eqref{E2} and \eqref{E3} in \eqref{eq: pb 1 carleman consequence} yield
\begin{equation}\label{eq: pb 1 carleman consequence 2}
\tau^3\| e^{\tau\varphi} \chi f \| ^2_{0,X}
+\tau\| e^{\tau\varphi} \nabla ( \chi f )\| ^2_{0,X} 
\lesssim    \alpha
 e^{\tau C_3 +2\alpha }\left(  \| u\|^2_{1,\omega'} +\| g\|^2_{0,\Omega}\right)   +  \alpha e^{C_1\tau+2\alpha  }\| u\|^2 _{0,\Omega} .
\end{equation}

Let  $\chi_0^2\in\Con_0^\infty(\omega) $ where $\chi_0=1$ on $\omega'$. We multiply~\eqref{eq: pb1 resolvent} by 
$\chi_0^2\overline u$ and we make an integration by parts. We  obtain
\[
\| \nabla u\|_{0,\omega'}^2\lesssim    \alpha \|u\|_{0,\omega}^2+ \| g\|_{0,\Omega}^2
\]
for which we deduce
\begin{equation}\label{eq: pb 1 carleman consequence 3}
\tau^3\| e^{\tau\varphi} \chi f \| ^2_{0,X}
+\tau\| e^{\tau\varphi} \nabla ( \chi f )\| ^2_{0,X} 
\lesssim    \alpha^2
 e^{\tau C_3 +2\alpha }\left(  \| u\|^2_{0,\omega} +\| g\|^2_{0,\Omega}\right) 
   +  \alpha e^{C_1\tau+2\alpha  }\| u\|^2 _{0,\Omega}.
\end{equation}

In the set $X\cap \{(s,x);\; |s|\le 1/2\}$, $\chi=1$ and 
\[
\varphi\ge e^{\lambda (-\beta/4+\min_{\Omega}\psi)}.
\] 
Then $e^{2\tau \varphi}\ge e^{\tau C_2}$, where
\[ 
C_2=2 e^{\lambda ( -\beta/4+\min_{\Omega}\psi ) }.
\] 
Fix then $\beta$ sufficiently large in such a way that $C_1<C_2<C_3$. From~\eqref{eq: pb 1 carleman consequence 3} we thus obtain
\begin{align*}
e^{\tau C_2+\alpha }\| u \| ^2_{0,\Omega}
\lesssim    \alpha^2
 e^{\tau C_3 +2\alpha }\left(  \| u\|^2_{0,\omega} +\| g\|^2_{0,\Omega}\right) 
 +  \alpha e^{C_1\tau+2\alpha  }\| u\|^2 _{0,\Omega}.
\end{align*}
Taking $\tau =\gamma \alpha=\gamma \sqrt{|\mu|}$ with $\gamma $ sufficiently large, there exist $C_4, C_5>0$ such that
\begin{align*}
\| u \| ^2_{0,\Omega}
\lesssim    
 e^{ C_4 \alpha }(  \| u\|^2_{0,\omega} +\| g\|^2_{0,\Omega}) 
 +   e^{-C_5 \alpha  }\| u\|^2_{0,\Omega}  .
\end{align*}
For $\alpha $ sufficiently large, we have 
\begin{align*}
\| u \| ^2_{0,\Omega}
\lesssim    
 e^{ C_4 \alpha }(  \| u\|^2_{0,\omega} +\| g\|^2_{0,\Omega}) .
 \end{align*}
 From~\eqref{eq: pb1 estimate on omega} we have
 \begin{align*}
\| u \| ^2_{0,\Omega} 
\le
K e^{ C_4 \alpha }\left(  \| u\|_{0,\Omega} \| g\|_{0,\Omega} +\| g\|^2_{0,\Omega}\right) .
 \end{align*}
As  
\[ Ke^{ C_4 \alpha } \| u\|_{0,\Omega} \| g\|_{0,\Omega}\le 
\| u\|_{0,\Omega} ^2/2+ (K^2/2)e^{ 2C_4 \alpha }\| g\|_{0,\Omega}^2,
\]
we obtain
 \begin{align*}
\| u \| ^2_{0,\Omega} 
\lesssim    
 e^{2 C_4 \alpha } \| g\|^2_{0,\Omega} ,
 \end{align*}
 which is exactly the expected resolvent estimate.
 \end{proof}
 
\subsubsection{Boundary damping}

This subsection is devoted to the proof of Theorem \ref{maintheorem1} when $\ell=2$ and $\ell=3$. We first restate for conveniance the result for $\ell=2$.

\begin{theorem} \label{th: pb2 carleman approach}
Let assumption $(A_d)$ holds. For every $\mu\in\mathbb{R}$, $A_2-i\mu$ is invertible and

\noindent
{\rm (i)} $\| (A_2-i\mu)^{-1}\|_{\mathscr{B}(V)}\le C e^{K\sqrt{|\mu|}}$, $\mu\in\mathbb{R}$, for  some constants $C>0$ and $K>0$,

\noindent
{\rm (ii)}  there exists a constant $C_1>0$, such that 
\[
\|e^{tA_2}u_0\|_{V}\le  \frac{C_1}{\ln^{2k}(2+t)}\| u_0\|_{D(A_2^k)},\;\; u_0\in D(A_2^k) .
\]
\end{theorem}

\begin{proof}
We are going to prove that $B=A_2$ obeys to the conditions of Theorem~\ref{th: resolvent and log decay} when $H=V$. As in the preceding proof, we solve the resolvent equation: for $g\in V$, find $u\in D(A_2)$ satisfying
\[
(A_2-i\mu)u=g.
\] 
Substituting $g$ by $-ig$, we are reduced the following equation: find $u\in D(A_2)$ so that
\begin{equation} \label{eq: pb2 resolvent}
\Delta_\mathbf{a}u-\mu u=g.
\end{equation}
Multiply this equation by $\overline u$ and  integrate over $\Omega$ in order to get 
\begin{equation*}
(\Delta_\mathbf{a}u| u)_{0,\Omega}-\mu(u|u)_{0,\Omega} =(g|u)_{0,\Omega},
\end{equation*}
In combination with \eqref{i1}, this identity yields
\begin{equation*}
-\|\nabla_\mathbf{a}u\|_{0,\Omega}^2-\mu\| u\|_{0,\Omega}^2+(\partial_{\nu_A}u|u)_{0,\Gamma_0}=(g|u)_{0,\Omega}.
\end{equation*}
As $\partial_{\nu_A} u=-id\Delta_\mathbf{a} u$ and $\Delta_\mathbf{a} u=\mu u +g$ on $\Gamma_0$, we have
\begin{equation}
\label{eq: pb2 inner product}
-\|\nabla_\mathbf{a}u\|_{0,\Omega}^2-\mu\| u\|_{0,\Omega}^2-i\mu \|\sqrt{d}u\|_{0,\Gamma_0} =(g|u)_{0,\Omega}+i(dg|u)_{0,\Gamma_0}.
\end{equation}
Taking the real part, we get
\begin{equation} \label{eq: pb 2 equation integration by parts}
-\|\nabla_\mathbf{a}u\|_{0,\Omega}^2-\mu\| u\|_{0,\Omega}^2=\Re( g|u)_{0,\Omega}  +\Re( idg|u) _{0,\Gamma_0}.
\end{equation}
For $\mu\ge 0$, we have
\[
\|\nabla_\mathbf{a}u \|_{0,\Omega}^2+\mu\| u\|_{0,\Omega}^2\le \|g\|_{0,\Omega}\|u\|_{0,\Omega}+\|d\|_\infty\|g\|_{0,\Gamma_0 }\|u\|_{0,\Gamma_0 }.
\]

We know that $\| \nabla_\mathbf{a} \cdot \|_{0,\Omega}$ is equivalent to the natural norm of $V$ induced by that on $H_0^1(\Omega)$. Therefore, the trace operator $\mathbf{tr}: V\rightarrow L^2(\Gamma _0 )$ is bounded when $V$ is endowed with norm $\|\nabla_\mathbf{a} \cdot \|_{0,\Omega}$. 

Thus, we have
\begin{align*}
\|\nabla_\mathbf{a}u\|_{0,\Omega}^2+\mu\| u\|_0^2&\le \varkappa_1^2 \|\nabla_\mathbf{a}g|\|_{0,\Omega}\|\nabla_\mathbf{a}u\|_{0,\Omega}+\|d\|_\infty\|\mathbf{tr}\|\||\nabla_\mathbf{a}g|\|_{0,\Omega}\||\nabla_\mathbf{a}u|\|_{0,\Omega}
\\
&\le\left( \varkappa_1^2 +\|d\|_\infty\|\mathbf{tr}\|^2\right)\|\nabla_\mathbf{a}g\|_{0,\Omega}\|\nabla_\mathbf{a}u\|_{0,\Omega},
\end{align*}
where $\|\mathbf{tr}\|$ denotes the norm of $\mathbf{tr}$ in $\mathscr{B}(V,L^2(\Gamma _0))$ and $\varkappa_1$ is the Poincar\'e constant of $V$.
In particular
\begin{equation} \label{eq: estimate for mu positive pb2}
\|\nabla_\mathbf{a}u\|_{0,\Omega}\le \left( \varkappa_1^2 +\|d\|_\infty\|\mathbf{tr}\|^2\right)\|\nabla_\mathbf{a}g\|_{0,\Omega}.
\end{equation}
This is nothing but the resolvent estimate for $\mu \ge 0$.

Let us now consider the case $\mu <0$. To this end, we firstly observe that $A_2-i\mu$  is injective. Indeed, take $g=0$ and then the imaginary part in \eqref{eq: pb2 inner product} to get $u=0$ on $\gamma_0$ yielding $\partial_\nu u=0$ on $\gamma_0$. Whence $u=0$ by the unique continuation property. Obviously, $g=0$ and $\mu =0$ entail $\nabla_\mathbf{a}u=0$ and then $u=0$. Next, as $A_2$ is invertible by the preceding step and $D(A_2)$ is compactly embedded in $V$ according to the elliptic regularity, $A_2^{-1}:V\rightarrow V$ is compact. Therefore $B_2= I-i\mu A_2^{-1}$, injective, is onto by Fredholm's alternative and hence $A_2-i\mu = A_2B_2$ is also onto.

To complete the proof, it remains to prove the resolvent estimate for $\mu <0$. As in the preceding proof it is enough to establish such an estimate for $|\mu|$ large.  To this end, taking one more time the imaginary part of~\eqref{eq: pb2 inner product}, we obtain
\begin{equation}\label{eq: pb2 Imaginary part formula}
-\mu \|\sqrt{d}u\|_{0,\Gamma_0}^2 =-\Re( ig|u) _{0,\Omega}+\Re(dg|u)_{0,\Gamma_0}.
\end{equation}

From~\eqref{eq: pb2 Imaginary part formula}, we get by using the continuity of the trace operator $\mathbf{tr}$ and $|\mu|\ge 1$,
\begin{equation} \label{est: a priori on u by damping}
d_0 \| u \|_{0,\gamma_0}^2\le \|d\|_\infty\|\mathbf{tr}\|^2\|\nabla_\mathbf{a}g\|_{0,\Omega}\|\nabla_\mathbf{a}u\|_{0,\Omega}.
\end{equation}

Next we proceed as in the preceding theorem. We first use a Carleman inequality to estimate $\|\nabla_\mathbf{a}u\|_{0,\Omega}$ by $\|u\|_{0,\Gamma_0}$.  Set $f(s,x)=e^{\alpha s}u(x)$, where $s\in (-2,2)$ and  $\alpha=\sqrt{-\mu }$. Then it is straightforward to check that $f$ satisfies
\begin{equation}
	\label{eq: pb2 equation adding variable}
Pf=\partial_s^2 f+\Delta_\mathbf{a} f= e^{s\alpha }g.
\end{equation}
Recall  that
\[
X=(-2,2)\times \Omega ,\;\; X_1=(-1,1)\times \Omega 
\]
and define
\[
\ell _0= (-1/2,1/2)\times \gamma_0.
\]

Let $\chi\in\Con_0^\infty(\mathbb{R})$, such that $\chi(s)=1$ for $|s|\le 3/4$ and $\chi(s)=0 $ for $ |s|\ge1$. We set 
\[
\varphi(s,x)=e^{\lambda(-\beta s^2+\psi(x))},
\]
 where $\psi$ is given by Proposition~\ref{prop: psi function boundary} with $\Lambda =\ell_0$ and $\beta>0$ is fixed in what follows.
 The function $-\beta s^2+\psi(x)$ has no critical point in $X$.
Then for $\lambda $ sufficiently large (but fixed from now on) $\varphi $ satisfies the sub-ellipticity condition of
Remark~\ref{rem: construction sub-ellipticity function}. 

In the rest of this proof $Q_1\lesssim Q_2$ means $Q_1\le CQ_2$, for some generic constant $C$, only depending on $n$, $\Omega$, $\mathbf{a}$ and $d$.

We can apply 
Theorem~\ref{th: Global Carleman estimate boundary Neumann Dirichlet}, 
with  $\chi f$ instead of $f$. As $\chi f$ satisfies Dirichlet boundary condition on 
$\Gamma_1$ and as $\partial_{\nu_\mathbf{a}}=\partial_\nu+i\mathbf{a}\cdot \nu$, we get
$\tau^3\| e^{\tau\varphi}f\|^2_{0,\ell _0}
+\tau\| e^{\tau\varphi}\partial_{\nu} f\|^2_{0,\ell_0} $ is equivalent to $\tau^3\| e^{\tau\varphi}f\|^2_{0,\ell _0}
+\tau\| e^{\tau\varphi}\partial_{\nu_\mathbf{a}} f\|^2_{0,\ell_0} $. Then

\begin{align}
\label{eq: carleman estimate type 1}
\tau^3\| e^{\tau\varphi}  \chi f \|_{0,X} ^2
&+\tau\| e^{\tau\varphi} \nabla (\chi f) \|_{0,X} ^2
  \lesssim \| e^{\tau\varphi}  P(\chi f) \|_{0,X}^2 
\\
&+\tau\| e^{\tau\varphi}\partial_{\nu_\mathbf{a}} f\|^2_{0,L_0 \setminus \ell_0}+\tau^3\| e^{\tau\varphi}f\|^2_{0,\ell _0}
+\tau\| e^{\tau\varphi}\partial_{\nu_\mathbf{a}} f\|^2_{0,\ell_0} . \notag
\end{align}

We recall the constants defined in the previous section,
\begin{align}
	\label{def: constants}
&C_1=  2e^{\lambda(  -9\beta/16+\sup_{\Omega} \psi)} , 
\\
& C_2=2e^{\lambda ( -\beta/4+\min_{\Omega}\psi ) },  \notag
\\
& C_3=2 e^{\lambda \sup_{\Omega}\psi(x)} .  \notag
\end{align}
Similarly to the previous section, we have 
\begin{align}  
	\label{est: right hand side Carleman type 1-1}
&  \| e^{\tau\varphi} P   (\chi f) \|_{0,X}^2   \lesssim  e^{2\alpha +C_3\tau } \|\nabla_\mathbf{a}g\|_{0,\Omega}^2 +  e^{2\alpha +C_1\tau } \|\nabla_\mathbf{a}u\|_{0,\Omega}^2,
\\
&  \| e^{\tau\varphi}f|^2_{0,\ell_0}  \lesssim  e^{2\alpha +C_3\tau } 
\| u \|_{0,\gamma_0}^2 \lesssim  e^{2\alpha +C_3\tau }  \|\nabla_\mathbf{a}g\|_{0,\Omega}\|\nabla_\mathbf{a}u\|_{0,\Omega},  \notag
\end{align}
from \eqref{est: a priori on u by damping}.
The two  other terms of the right hand side of \eqref{eq: carleman estimate type 1}  may be estimated by 
$\| e^{\tau\varphi}\partial_{\nu_\mathbf{a}} f\|^2_{0,L_0}$.
We have
\begin{equation*}
(\partial_{\nu_A} f )_{|L_0}=e^{\alpha s}(\partial_{\nu_A} u )_{|\Gamma_0}=-ide^{\alpha s}(\Delta_\mathbf{a} u)_{|\Gamma_0}
=-ide^{\alpha s}( g+\mu u)_{|\Gamma_0}.
\end{equation*}
Whence
\begin{align}
\label{est: right hand side Carleman type 1-2}
\| e^{\tau\varphi}\partial_{\nu_\mathbf{a}} f\|^2_{0,L_0}
&\lesssim   e^{2\alpha +C_3\tau }( \|g\|^2 _{0,\Gamma_0}+|\mu|^2\|d^{1/2}u\|^2_{0,\Gamma_0}) 
\\
&\lesssim   e^{2\alpha +C_3\tau }\left( \|\nabla_\mathbf{a}g\|_{0,\Omega}^2+\alpha^4    \|\nabla_\mathbf{a}g\|_{0,\Omega}\|\nabla_\mathbf{a}u\|_{0,\Omega}\right) .   \notag
\end{align}

On the other hand, it is straightforward to check 
\begin{align}
	\label{est: left hand side carleman type 1}
\tau^3\| e^{\tau\varphi}  \chi f \|_{0,X} ^2
+\tau\| e^{\tau\varphi} \nabla (\chi f) \|_{0,X} ^2& \gtrsim \tau^3\| e^{\tau\varphi}   f \|_{0,(-1/2,1/2)\times \Omega} ^2
\\
&\hskip 2cm +\tau\| e^{\tau\varphi} \nabla f \|_{0,(-1/2,1/2)\times \Omega} ^2\notag
\\
&\gtrsim 
 e^{\alpha +\tau C_2}( \|  u \|_{0,\Omega} ^2
+\|  \nabla u \|_{0,\Omega} ^2).  \notag
\end{align}
Inequalities \eqref{eq: carleman estimate type 1} and \eqref{est: right hand side Carleman type 1-1} to \eqref{est: left hand side carleman type 1} yield
\begin{align*}
 &e^{\alpha +\tau C_2} \| \nabla_\mathbf{a}u\|^2_{0,\Omega}
 \\
 &\qquad \lesssim  e^{2\alpha +C_3\tau }(   \| \nabla_\mathbf{a}g\|_{0,\Omega}^2 + \alpha^4 \| \nabla_\mathbf{a}g\|_{0,\Omega}  \| \nabla_\mathbf{a}u\|_{0,\Omega})
 + \alpha e^{2\alpha +C_1\tau } \| \nabla_\mathbf{a}u\| _{0,\Omega}^2.
\end{align*}
As we have done in the preceding proof, taking $\beta$ sufficiently large, we have $C_1<C_2<C_3$  and,  for $\tau=\gamma \alpha$ with $\gamma $ sufficiently 
large, we find $C_4>0$  and $C_5>0$ so that
\begin{align*}
 \|\nabla_\mathbf{a}u\|_{0,\Omega}^2\le  C e^{C_4\alpha}\left( \|\nabla_\mathbf{a}g\|_{0,\Omega}^2 + \alpha^4 \|\nabla_\mathbf{a}g\|_{0,\Omega}\|\nabla_\mathbf{a}u\|_{0,\Omega}\right)
 + C e^{-C_5\alpha } \|\nabla_\mathbf{a}u\|_{0,\Omega}^2.
\end{align*}
Choose $\alpha$ sufficiently large in such a way that $C  e^{-C_5\alpha } \le 1/4$. Then 
\begin{equation*}
 C\alpha^4 e^{C_4\alpha} \|\nabla_\mathbf{a}g\|_{0,\Omega}  \|\nabla_\mathbf{a}u\|_{0,\Omega} 
\le C^2 \alpha^8 e^{2C_4\alpha} \|\nabla_\mathbf{a}g\|_{0,\Omega}^2+  \|\nabla_\mathbf{a}u\|_{0,\Omega}^2/2.
\end{equation*}
The last two estimates entail
\begin{equation}
 \|\nabla_\mathbf{a}u\|_{0,\Omega}^2\le  C e^{C \alpha}  \|\nabla_\mathbf{a}g\|_{0,\Omega}^2 .
\end{equation}
The proof is then complete.
\end{proof}


We move now to the system \eqref{schfront1b} for which we aim to prove Theorem \ref{maintheorem1} for $\ell=3$. We restate here for convenience this result.

\begin{theorem} \label{th: pb3 carleman approach}
Under assumption $(A_d)$, for every $\mu\in\mathbb{R}$, $A_3-i\mu$ is invertible and

\noindent
{\rm (i)} $ \| (A_3-i\mu)^{-1}\|_{\mathscr{B}(L^2(\Omega))}\le C e^{K\sqrt{|\mu|}}$,
$\mu\in\mathbb{R}$, for some constants $C>0$ and $K$,

\noindent
{\rm (ii)} there exists a constant $C_1>0$, such that 
\[
 \|  e^{tA_3}u_0\|_{0,\Omega}\le 
\frac{C_1}{\ln^{2k}(2+t)}\| u_0\|_{D(A_3^k)},\;\; u_0\in D(A_3^k) .
\]
\end{theorem}

\begin{proof}
As in the preceding two proofs, we first solve the resolvent equation: for $g\in L^2(\Omega )$, find $u\in D(A_3)$ so that
\begin{equation} \label{eq: pb3 resolvent}
\Delta_\mathbf{a}u-\mu u=g.
\end{equation}

With the help of identity \eqref{i1}, we get
\begin{equation*}
-\| \nabla_\mathbf{a}u\|_{0,\Omega}^2-\mu\| u\|_{0,\Omega}^2+(\partial_{\nu_A} u|u)_{0,\Gamma_0}=(g|u)_{0,\Omega}.
\end{equation*}
As $\partial_{\nu_A} u=idu$, we have
\begin{equation}
\label{eq: pb3 inner product}
-\| \nabla_\mathbf{a}u\|_{0,\Omega}^2-\mu\| u\|_{0,\Omega}^2+i\|\sqrt{d}u\|_{0,\Gamma_0}^2=(g|u)_{0,\Omega}.
\end{equation}
Take the real part of each side in order to derive
\begin{equation} \label{eq: pb 3 equation integration by parts}
-\| \nabla_\mathbf{a}u\|_{0,\Omega}^2-\mu\| u\|_{0,\Omega}^2=\Re( g|u)_{0,\Omega}.
\end{equation}
When $\mu\ge0$, we obtain 
\begin{equation}
	\label{eq: estimate for mu positive}
\| \nabla_\mathbf{a}u\|_{0,\Omega}\le \| g\|_{0,\Omega}.
\end{equation}  
This and Poincar\'e inequality on $V$ imply the resolvent estimate when $\mu \ge 0$.

When $\mu <0$, we can repeat the argument we used for $A_2$. That is $A_3-i\mu$ will be invertible if it is injective. Here again the fact that $A_3-i\mu$ is injective follows from a unique continuation property. 

Next assume that $\mu <0$. We get by taking the imaginary part of each side of~\eqref{eq: pb3 inner product}
\[
\|\sqrt{d}u\|_{0,\Gamma_0}^2=-\Re( ig|u)_{0,\Omega} .
\]
Hence 
\begin{equation}
	\label{est: a priori on u by damping type 2}
d_0\|u\|_{0,\gamma_0}^2\le \| g\|_{0,\Omega} \| u\|_{0,\Omega}.
\end{equation}

In this proof $\lesssim$ has the same meaning as in the proof of Theorem \ref{th: pb2 carleman approach}.

With the notations of the preceding proof, we have 
\begin{align}  
	\label{est: right hand side Carleman type 2}
&  \| e^{\tau\varphi} P   (\chi f) \|_{0,X}^2   \lesssim  e^{2\alpha +C_3\tau } \| g\|_{0,\Omega}^2+ \alpha e^{2\alpha +C_1\tau } \| u\| _{1,\Omega}^2
\\
&  \| e^{\tau\varphi}f\|^2_{0,L_0}  \lesssim  e^{2\alpha +C_3\tau } 
\| u \|_{L^2(\gamma_0)}^2 \lesssim  e^{2\alpha +C_3\tau }  \| g\|_{0,\Omega}  \| u\|_{0,\Omega},  \notag
\end{align}
where we used \eqref{est: a priori on u by damping type 2}.

The two other  terms of the right hand side of \eqref{eq: carleman estimate type 1} are estimated by 
$\| e^{\tau\varphi}\partial_{\nu_\mathbf{a}} f\|^2_{0,L_0}$. We have
\begin{equation*}
(\partial_{\nu_\mathbf{a}} f )_{| L_0}=e^{\alpha s}(\partial_{\nu_\mathbf{a}} u )_{|\Gamma_0}=ide^{\alpha s}u_{|\Gamma_0}.
\end{equation*}
Whence, using \eqref{est: a priori on u by damping type 2}, we get
\begin{align}
\label{est: right hand side Carleman type 2-2}
\| e^{\tau\varphi}\partial_{\nu_\mathbf{a}} f\|^2_{0,L_0}
&\lesssim   e^{2\alpha +C_3\tau }  \|\sqrt{d}u\|^2_{0,\Gamma_0}
\\
&\lesssim   e^{2\alpha +C_3\tau }  \| g\|_{0,\Omega}\| u\|_{0,\Omega}   .   \notag
\end{align}

On the other hand,
\begin{align}
	\label{est: left hand side carleman type 2}
\tau^3\| e^{\tau\varphi}  \chi f \|_{0,X} ^2
+\tau\| e^{\tau\varphi} \nabla (\chi f) \|_{0,X} ^2& \gtrsim \tau^3\| e^{\tau\varphi}   f \|_{0,(-1/2,1/2)\times \Omega} ^2
\\
&\hskip 2cm+\tau\| e^{\tau\varphi} \nabla f \|_{0,(-1/2,1/2)\times \Omega} ^2\notag
\\
&\gtrsim 
 e^{\alpha +\tau C_2}( \|  u \|_{0,\Omega} ^2
+|  \nabla u \|_{0,\Omega} ^2).  \notag
\end{align}
Estimates \eqref{eq: carleman estimate type 1} and \eqref{est: right hand side Carleman type 2} to \eqref{est: left hand side carleman type 2},
imply
\begin{align}
 e^{\alpha +\tau C_2} \| u\|^2_{1,\Omega}    \lesssim  e^{2\alpha +C_3\tau }(   \| g\|_{0,\Omega} ^2 + \| g\|_{0,\Omega}  \| u\|_{0,\Omega} )
 +  e^{2\alpha +C_1\tau } \| u\| _{1,\Omega} ^2.
\end{align}
Similarly to the proof of the preceding  theorem, we can take $\beta$  large enough in order to ensure that $C_1<C_2<C_3$ and, for $\tau=\gamma \alpha$ with $\gamma $ sufficiently large, there exist $C_4>0$ and $C_5>0$ so that
\begin{align*}
 \| u\|^2_{1,\Omega}  \le  C e^{C_4\alpha}(   \| g\|_{0,\Omega}  ^2 + \| g\|_{0,\Omega} \| u\|_{0,\Omega})
 + C e^{-C_5\alpha } \| u\| _{1,\Omega}  ^2.
\end{align*}
Pick $\alpha$  large enough in such a way that $C  e^{-C_5\alpha } \le 1/4$. Then 
\begin{equation*}
 C  e^{C_4\alpha} \| g\|_{0,\Omega} \| u\|_{0,\Omega} 
\le C^2  e^{2C_4\alpha} \| g\|_{0,\Omega}^2+   \| u\|_{0,\Omega}  ^2/2.
\end{equation*}
The two last estimates yield
\begin{equation}
\| u\|_{0,\Omega}^2\le  \| u\|_{1,\Omega}  ^2\le  C e^{C \alpha}  \| g\|_{0,\Omega}^2 ,
\end{equation}
That is we proved the resolvent estimate for $\mu <0$.
\end{proof}

\section{Exponential stabilization}
\setcounter{equation}{0}

\subsection{Observability inequalities}
In this section, we use the following notation
\[
Q=\Omega \times (0,T),\;\; \Sigma =\Gamma \times (0,T)\;\; \mbox{and}\;\; \Sigma_j =\Gamma_j \times (0,T),\; j=0,1.
\]

Following Lions and Magenes notation, the anisotropic Sobolev space $H^{2,1}(Q)$ is given by
\[
H^{2,1}(Q)=L^2((0,T),H^2(\Omega ))\cap H^1((0,T),L^2(\Omega )).
\]

We use frequently in the sequel the following Green's formula
\begin{equation}\label{eq3.1}
((\partial _j+ia_j)u|v)_{0,\Omega}=-( u|(\partial _j+ia_j)u)_{0,\Omega}+(u|v\nu_j)_{0,\Gamma} .
\end{equation}

The following proposition is a key tool in the multiplier method.

\begin{proposition}\label{proposition3.1}
Let $\aleph \in C^2(\overline{Q},\mathbb{R}^n)$, $u\in H^{2,1}(Q)$ and set
\[
f=i\partial _t u+\Delta _\mathbf{a}u.
\]
Then
\begin{align*}
\langle  \nabla _\mathbf{a}u\cdot \nu |\aleph \cdot \nabla _\mathbf{a}u\rangle_{0,\Sigma} &-
\frac{1}{2}(  |\nabla _\mathbf{a}u|^2|\aleph \cdot \nu)_{0,\Sigma}
\\
&+\frac{1}{2}(\mbox{div}(\aleph)u|\nabla _\mathbf{a}u\cdot \nu )_{0,\Sigma}-\frac{i}{2}(u(\aleph \cdot \nu )|\partial _tu)_{0,\Sigma}
\\
&= \langle D\aleph \nabla _\mathbf{a}u| \nabla _\mathbf{a}u\rangle_{0,Q}+\frac{1}{2}(u\nabla 
\mbox{div}(\aleph )|\nabla _\mathbf{a}u)_{0,Q} 
\\
&+\frac{i}{2}(u\partial _t\aleph |\nabla _\mathbf{a}u)_{0,Q} - \frac{i}{2} \left[(u\aleph |
\nabla _\mathbf{a}u)_{0,\Omega}\right]_0^Tdx
\\
&+\langle f\aleph |\nabla _\mathbf{a}u\rangle_{0,Q}+\frac{1}{2}(\mbox{div}(\aleph)u|f)_{0,Q}.
\end{align*}
Here $D\aleph =(\partial_k\aleph_\ell )$ is the Jacobian matrix of $\aleph$.
\end{proposition}

\begin{proof} For simplicity sake's, we use in this proof the following temporary notation
\[
d_j=\partial _j+ia_j\;\; \mbox{and}\;\; \overline{d_j}=\partial _j-ia_j.
\]

\noindent
\textbf{First step.} We prove
\begin{align}
&\langle\Delta _\mathbf{a}u|\aleph\cdot \nabla _\mathbf{a}u\rangle_{0,Q} =-\langle
D\aleph \nabla _\mathbf{a}u | \nabla _\mathbf{a}u\rangle_{0,Q} \label{eq3.2}
\\
&\hskip 2cm +\frac{1}{2}(|\nabla _\mathbf{a}u|^2|\mbox{div}(\aleph))_{0,Q}
 -\frac{1}{2}( |\nabla _\mathbf{a}u|^2|\aleph \cdot \nu )_{0,\Sigma} +\langle  \nabla _\mathbf{a}u\cdot \nu |\aleph \cdot \nabla _\mathbf{a}u\rangle_{0,\Sigma}. \nonumber
\end{align}
From Green's formula \eqref{eq3.1}, we have
\begin{align}
(\Delta _\mathbf{a}u|\aleph\cdot \nabla _\mathbf{a}u)_{0,\Omega}&=\sum_{j,k=1}^n(d_j^2u
\aleph _k|d_ku)_{0,\Omega} \label{eq3.3}
\\
&=-\sum_{j,k=1}^n(d_ju|d_j(\aleph _kd_ku))_{0,\Omega}+\sum_{j,k=1}^n(d_ju\nu _j|\aleph _kd_ku)_{0,\Gamma} \nonumber
\\
&=-\sum_{j,k=1}^n(d_ju|d_j(\aleph _kd_ku))_{0,\Omega}+(\nabla _\mathbf{a}u\cdot \nu |\aleph \cdot \nabla _\mathbf{a}u)_{0,\Gamma} .\nonumber
\end{align}
Elementary calculations show
\[
d_j(\aleph _kd_ku)=\partial _j\aleph _kd_ku+\aleph _kd_jd_ku.
\]
Therefore
\[
(d_ju|d_j(\aleph _kd_ku))_{0,\Omega}=(\partial _j\aleph _kd_ju|d_ku)_{0,\Omega}
+(d_ju\aleph _k|d_jd_ku)_{0,\Omega}.
\]
Hence
\begin{equation}\label{eq3.4}
\sum_{i,k=1}^n(d_ju|d_j(\aleph _kd_ku))_{0,\Omega}=(D\aleph \nabla _\mathbf{a}u| \nabla _\mathbf{a}u)_{0,\Omega}
+\sum_{i,k=1}^n(d_ju\aleph _k|d_jd_ku)_{0,\Omega}.
\end{equation}

Introduce the auxiliary function $v_j=d_ju$. Then
\[
d_ju\overline{d_j}\, \overline{d_ku}=v_j\overline{d_k}\overline{v_j}=v_j\partial _k\overline{v_j}-ia_j
|v_j|^2
\]
and then
\[
\Re [d_ju\overline{d_j}\, \overline{d_ku}]=\Re (v_j\partial _k\overline{v_j})=
\frac{1}{2}(v_j\partial _k\overline{v_j}+\overline{v_j}\partial _kv_j)=\frac{1}{2}\partial_k|v_j|^2=
\frac{1}{2}\partial_k|d_ju|^2.
\]
Whence
\begin{align*}
\sum_{i,j=1}^n\langle\aleph _kd_ju|d_jd_ku\rangle_{0,\Omega}&=\frac{1}{2}( 
\nabla |\nabla_\mathbf{a}u|^2|\aleph)_{0,\Omega}
\\
&= -\frac{1}{2}(|\nabla _\mathbf{a}u|^2|\mbox{div}(\aleph ))_{0,\Omega}+\frac{1}{2}(|\nabla _\mathbf{a}u|^2|\aleph \cdot \nu )_{0,\Gamma}.
\end{align*}
This and \eqref{eq3.4} lead
\begin{align*}
\sum_{j,k=1}^n\langle \aleph _kd_ju|d_j(\aleph _kd_ku))_{0,\Omega}&=\langle 
D\aleph \nabla _\mathbf{a}u|\nabla _\mathbf{a}u)_{0,\Omega}
\\
&-\frac{1}{2}(|\nabla _\mathbf{a}u|^2|\mbox{div}(\aleph ))_{0,\Omega}+\frac{1}{2}(|\nabla _\mathbf{a}u|^2|\aleph \cdot \nu )_{0,\Gamma}.
\end{align*}
Combine this identity with the real part of \eqref{eq3.3} and  integrate with respect to  $t$ in order to get the expected  identity.

\noindent
\textbf{Second step.} We have
\begin{align*}
2\Re \left[i\partial _t u\left(\aleph \cdot \overline{\nabla _\mathbf{a}u}\right)\right]&=i\partial _t u\left(\aleph \cdot 
\overline{\nabla _\mathbf{a}u}\right)-i\partial _t\overline{u}\left(\aleph \cdot\nabla _\mathbf{a}u\right)
\\
&=
i\left[\partial _t u\left(\aleph \cdot \overline{\nabla u}\right)-\partial _t\overline{u}\left(\aleph \cdot \nabla u\right)\right]+(\aleph \cdot \mathbf{a})
(\partial _tu\overline{u}+u\partial _t\overline{u})
\\
&=
i\left[\partial _t u\left(\aleph \cdot \overline{\nabla u}\right)-\partial _t\overline{u}\left(\aleph \cdot\nabla u\right)\right]
+(\aleph \cdot \mathbf{a})\partial _t|u|^2.
\end{align*}
An integration by parts with respect to $t$ gives
\[
(\aleph \cdot \mathbf{a}| \partial _t|u|^2)_{0,(0,T)}=-(\partial _t\aleph \cdot \mathbf{a} ||u|^2)_{0,(0,T)} +\left[(\aleph \cdot \mathbf{a})
|u|^2\right]_0^T.
\]
Therefore
\begin{align}
\langle i\partial _t u\aleph |\nabla _\mathbf{a}u)_{0,Q} &=
\frac{i}{2}\left[(\partial _t u\aleph |\nabla u)_{0,Q}- (\nabla u|\partial _t u\aleph )_{0,Q}\right] \label{eq3.5}
\\
& -\frac{1}{2}(\partial _t\aleph \cdot \mathbf{a}||u|^2)_{0,Q}+\frac{1}{2} \left[(\aleph \cdot \mathbf{a}|
|u|^2)_{0,\Omega}\right]_0^T.\nonumber
\end{align}
Next we calculate the first term in the right hand side of the identity above. Integrating with respect
to $t$, we find
\[
(\partial _t u|\aleph \cdot \nabla u)_{0, (0,T)}=-(u|\partial _t\aleph \cdot
\nabla u)_{0,(0,T)}-(u|\aleph \cdot \partial _t\nabla u)_{0,(0,T)}+\left[u(\aleph\cdot
\overline{\nabla u})\right]_0^T.
\]

On the other hand, Green's formula yields
\[
(u\aleph |\partial _t\nabla u)_{0,Q}=-(\mbox{div}(\aleph ) u|\partial _t
u)_{0,Q}-(\aleph \cdot \nabla u|\partial _tu)_{0,Q}+(
(\aleph \cdot \nu)u |\partial_tu)_{0,\Sigma}.
\]
Hence
\begin{align}
\frac{i}{2}\left[(\partial _t u\aleph |\nabla u)_{0,Q}- (\nabla u|\partial _t u\aleph )_{0,Q}\right] &=-\frac{i}{2}(u\partial _t\aleph | \nabla u)_{0,Q}
+\frac{i}{2} (\mbox{div}(\aleph ) u|\partial _tu)_{0,Q} \label{eq3.6}
\\
&+\frac{i}{2} \left[u\aleph|\nabla u)_{0,\Omega}\right]_0^Tdx-\frac{i}{2}
(u(\aleph \cdot \nu)|\partial _tu)_{0,\Sigma}. \nonumber
\end{align}
 
\noindent
\textbf{Step three.} We calculate the term $(\mbox{div}(\aleph ) u|\partial _tu)_{0,Q}$ in \eqref{eq3.6}.  Using $i\partial _tu=-\Delta _\mathbf{a}u+f$, we find

\begin{equation}\label{eq3.7}
i(\mbox{div}(\aleph ) u|\partial _tu)_{0,Q}=(\mbox{div}(\aleph ) u|\Delta _\mathbf{a}u)_{0,Q}-(\mbox{div}(\aleph ) u|f)_{0,Q}.
\end{equation}
But
\begin{align}
(\mbox{div}(\aleph ) u|\Delta _\mathbf{a}u)_{0,Q}&=\sum_{j=1}^n(\mbox{div}(\aleph ) u|d_jd_ju )_{0,Q}\label{eq3.8}
\\
&=-\sum_{j=1}^n(d_j(\mbox{div}(\aleph ) u)|d_ju)_{0,Q}+\sum_{j=1}^n(\mbox{div}(\aleph ) u\nu _j| d_ju)_{0,\Sigma}\nonumber
\\
&=-\sum_{j=1}^n(\mbox{div}(\aleph )d_ju|d_ju)_{0,Q}-\sum_{j=1}^n(\partial _j\mbox{div}(\aleph ) u|d_ju)_{0,Q}\nonumber
\\
&\hskip 2cm
+\sum_{j=1}^n(\mbox{div}(\aleph ) u\nu _j |d_ju)_{0,\Sigma} \nonumber
\\
&= -(\mbox{div}(\aleph )||\nabla _\mathbf{a}u|^2)_{0,Q}-(u\nabla (\mbox{div}(\aleph ))|
\nabla _\mathbf{a}u)_{0,Q}
\nonumber
\\
&\hskip 2cm
+(\mbox{div}(\aleph ) u|\nabla _\mathbf{a}u\cdot \nu)_{0,\Sigma}.
\nonumber
\end{align}
A combination of \eqref{eq3.5} to \eqref{eq3.8} entails
\begin{align}
\langle i\partial _tu\aleph |\nabla _\mathbf{a}u)_{0,Q}=&
-\frac{i}{2}(\partial _t\aleph |\nabla u)_{0,Q}-\frac{1}{2}(\mbox{div}(\aleph)|
|\nabla _\mathbf{a}u|^2)_{0,Q}  \label{eq3.9}
\\
& -\frac{1}{2}(u\nabla (\mbox{div}(\aleph))|\nabla _Au)_{0,Q}-\frac{1}{2}
( \mathbf{a}\cdot \partial _t \aleph ||u|^2)_{0,Q} \nonumber
\\
&+\frac{1}{2}(\mbox{div}(\aleph)u|f)_{0,Q}\nonumber
\\
&+\frac{i}{2}[(u\aleph |\nabla u)_{0,\Omega}]_0^Tdx+\frac{1}{2} 
[(u\aleph |u\mathbf{a})_{0,\Omega}]_0^Tdx\nonumber
\\
& +\frac{1}{2}(\mbox{div}(\aleph)u|\nabla _\mathbf{a}u\cdot \nu )_{0,\Sigma}
-\frac{i}{2}(u(\aleph \cdot \nu )|\partial_tu)_{0,\Sigma}. \nonumber
\end{align}
We put together the first and the fourth terms of the right hand side of this inequality. We obtain
\[
-\frac{i}{2}(u\partial _t\aleph |\nabla u)_{0,Q}-\frac{1}{2}
(\partial _t \aleph u|u\mathbf{a})_{0,Q} = -\frac{i}{2}(u\partial _t\aleph | \nabla _\mathbf{a}u)_{0,Q}.
\]
Similarly, we put together the sixtieth and the ninetieth terms for the right hand of \eqref{eq3.9}. We get
\[
\frac{i}{2}\left[(u\aleph |\nabla u)_{0,\Omega}\right]_0^Tdx+\frac{1}{2}( 
\left[(u\aleph |u\mathbf{a})_{0,\Omega}\right]_0^Tdx=\frac{i}{2}\left[(u\aleph |\nabla _\mathbf{a}u)_{0,\Omega}\right]_0^Tdx.
\]
Then \eqref{eq3.9} becomes

\begin{align}
\langle i\partial _tu\aleph |\nabla _\mathbf{a}u)_{0,Q}=&
-\frac{i}{2}(\partial _t\aleph |\nabla_\mathbf{a} u)_{0,Q}-\frac{1}{2}(\mbox{div}(\aleph)|
|\nabla _\mathbf{a}u|^2)_{0,Q}  \label{eq3.10}
\\
& -\frac{1}{2}(u\nabla (\mbox{div}(\aleph))|\nabla _\mathbf{a}u)_{0,Q}\nonumber
\\
&+\frac{1}{2}(\mbox{div}(\aleph)u|f)_{0,Q}\nonumber
\\
&+\frac{i}{2}[(u\aleph |\nabla_\mathbf{a} u)_{0,\Omega}]_0^Tdx \nonumber
\\
& +\frac{1}{2}(\mbox{div}(\aleph)u|\nabla _\mathbf{a}u\cdot \nu )_{0,\Sigma}
-\frac{i}{2}(u(\aleph \cdot \nu )|\partial_tu)_{0,\Sigma}. \nonumber
\end{align}

\noindent
\textbf{Step four.} We complete the proof by noting the expected identity follows from \eqref{eq3.2}, \eqref{eq3.10} and 
\[
\Re \left[i\partial _tu\aleph \cdot \overline{\nabla _\mathbf{a}u}\right]+\Re 
\left[\Delta _\mathbf{a}u\aleph \cdot \overline{\nabla _Au}\right]=\Re 
\left[f\aleph \cdot \overline{\nabla _\mathbf{a}u}\right].
\]
\end{proof}

Bearing in mind that $D(A_0)=H_0^1(\Omega)\cap H^2(\Omega )$, we derive that, for $u_0\in D(A_0)$, 
\[
u(t)=e^{tA_0}u_0\in C([0,T],D(A_0))\cap C^1([0,T],L^2(\Omega ))\subset H^{2,1}(Q).
\]

\begin{corollary}\label{corollary3.1}
There exists a constant $C=C_1+C_2\sqrt{T}>0$, the constants $C_1$ and $C_2$ only depend on $\Omega$, so that, for any $u_0\in D(A_0)$ and $u(t)=e^{tA_0}u_0$, we have
\begin{equation}\label{eq3.11}
\|\nabla_\mathbf{a}u\|_{0,\Sigma} \le C\|\nabla_\mathbf{a}u_0\|_{0,\Omega}.
\end{equation}
\end{corollary}

\begin{proof}
We firstly note that according to \cite[Lemma 3.2]{BC},
\begin{equation}\label{eq3.12}
\|u(\cdot ,t)\|_{0,\Omega}=\|u_0\|_{0,\Omega} \;\; \mbox{and}\;\; \|\nabla_\mathbf{a}u(\cdot ,t)\|_{0,\Omega}=\|\nabla_\mathbf{a}u_0\|_{0,\Omega},\;\; 0\le t\le T.
\end{equation}
Let us choose $\aleph \in C^\infty (\overline{\Omega},\mathbb{R}^n)$ as an extension of $\nu$. In that case the left hand side of the identity in Proposition \ref{proposition3.1} is equal to the square of the left hand side of \eqref{eq3.11}. While the right hand of the identity in Proposition \ref{proposition3.1} is bounded by the square of the right hand side of \eqref{eq3.11}. This a consequence of Cauchy-Schwarz's inequality and \eqref{eq3.12}.
\end{proof}

In the rest of this section, $x_0\in \mathbb{R}^n$ is fixed, $m=m(x)=x-x_0$, $x\in \mathbb{R}^n$ and
\[
\Gamma_0=\{ x\in \Gamma ;\; m(x)\cdot \nu (x)>0\}.
\]

Observe that in the present case the condition $\overline{\Gamma_0}\cap\overline{\Gamma_1}=\emptyset$ is satisfied for instance if $\Omega =\Omega_0\setminus \Omega_1$, with $\Omega_1\Subset \Omega_0$, $\Omega_j$ star-shaped with respect to $x_0\in \Omega_1$ and $\Gamma_j=\partial \Omega_j$, $j=0,1$.

We now sketch the proof of the following observability inequality announced in the introduction.

\begin{proposition}\label{proposition3.2}
There exists a constant $C>0$,  only depending on $\Omega$ and $T$, so that,  for any $u_0\in D(A_0)$ and $u(t)=e^{tA_0}u_0$, we have
\begin{equation}\label{eq3.13}
\|\nabla_\mathbf{a}u_0\|_{0,\Omega}  \le C\|\nabla_\mathbf{a}u\|_{0,\Sigma_0}=C\|\partial_{\nu_\mathbf{a}}u\|_{0,\Sigma_0}.
\end{equation}
\end{proposition}

\begin{proof}[Sketch of the proof]
Take $\aleph =m$ in the identity of Proposition \ref{proposition3.1}. We get
\[
(m\cdot \nu||\partial_{\nu_\mathbf{a}} u|^2)_{0,\Sigma}=\|\nabla_\mathbf{a}u\|_{0,Q}^2-\frac{i}{2}\left[(um|\nabla_\mathbf{a}u)_{0,\Omega}\right]_0^T.
\]
Whence, in light of \eqref{eq3.12},
\[
T\|\nabla_\mathbf{a} u_0\|_{0,\Omega}^2\le (m\cdot \nu||\partial_{\nu_\mathbf{a}} u|^2)_{0,\Sigma_0}+\frac{1}{2}\left|\left[(um|\nabla_\mathbf{a}u)_{0,\Omega}\right]_0^T\right|.
\]
But, for $0<\epsilon <T$, there exists a constant $C_\epsilon >0$, independent on $T$, so that
\[
\frac{1}{2}\left|\left[(um|\nabla_\mathbf{a}u)_{0,\Omega}\right]_0^T\right|\le \epsilon \|\nabla_\mathbf{a} u_0\|_{0,\Omega}^2+C_\epsilon\|u_0\|_{0,\Omega}^2,
\]
where we used again \eqref{eq3.12}. Hence
\[
(T-\epsilon )\|\nabla_\mathbf{a} u_0\|_{0,\Omega}^2\le \|m\|_\infty \|\partial_{\nu_\mathbf{a}} u\|_{0,\Sigma_0}^2+C_\epsilon\|u_0\|_{0,\Omega}^2.
\]
As $\|\nabla_\mathbf{a} \cdot \|_{0,\Omega}$ and $\|\nabla\cdot \|_{0,\Omega}$  are equivalent on $H_0^1(\Omega )$, we can repeat the compactness argument in \cite[Proposition 2.1]{mach} to complete the proof.
\end{proof}

We also sketch the proof of the observability inequality with interior control. We restate here for convenience this result.

Recall that for this result $\omega$ is a neighborhood of $\Gamma_0$ in $\Omega$ so that $\overline{\omega}\cap \Gamma_1=\emptyset$.

\begin{proposition}\label{proposition3.3}
There exists a constant $C>0$,  only depending on $\Omega$, $T$, $\Omega$ and $\Gamma_0$, so that, for any $u_0\in D(A_0)$ and $u(t)=e^{tA_0}u_0$, we have
\begin{equation}\label{eq3.13bis}
\|u_0\|_{0,\Omega}  \le C\|u\|_{0,Q_\omega}.
\end{equation}
Here $Q_\omega =\omega \times (0,T)$.
\end{proposition}

\begin{proof}[Sketch of the proof]
Fix $0<\delta <T$. Let $\nu_e \in C^\infty (\overline{\Omega} ,\mathbb{R}^n)$ be an extension of $\nu$, $0\le \phi \in C^\infty_0 (0,T)$ satisfying $\phi =1$ in $[\delta ,T-\delta]$, and $\psi \in C_0^\infty (\mathbb{R}^n)$ so that $\mbox{supp}(\psi )\cap \Omega \subset \widetilde{\omega}$ and $\psi =1$ on $\Gamma_0$. 

\smallskip
We have from Proposition \ref{proposition3.2} with $\aleph=\nu_e\phi \psi$, in which $(0,T)$ is substituted by $(\delta ,T-\delta )$,
\begin{equation}\label{eq3.14}
\|\nabla_\mathbf{a}u_0\|_{0,\Omega}=\|\nabla_\mathbf{a}u(\cdot ,\delta )\|_{0,\Omega}\le C\|\partial_{\nu_\mathbf{a}}u\|_{0,\Gamma _0\times (\delta ,T-\delta )}\le C\|(\aleph \cdot \nu )\partial_{\nu_\mathbf{a}}u\|_{0,\Sigma}.
\end{equation}
Let $\tilde{\omega}$ be a neighborhood of $\omega$ in $\Omega$ satisfying $\overline{\tilde{\omega}}\cap \Gamma_1=\emptyset$. As in the proof of Corollary \ref{corollary3.1}, we obtain by applying Proposition \ref{proposition3.1}, where $Q_{\tilde{\omega}}=\widetilde{\omega}\times (0,T)$,
\begin{equation}\label{eq3.15}
C\|(\aleph \cdot \nu )\partial_{\nu_\mathbf{a}}u\|_{0,\Sigma}\le \|\nabla_\mathbf{a}u\|_{0,Q_{\tilde{\omega}}}+\|u\|_{0,Q_{\tilde{\omega}}}.
\end{equation}

On the other hand, using $\Delta_\mathbf{a}u(\cdot ,t)=-i\partial_tu(\cdot ,t)$ in $\Omega$ and Caccioppoli's inequality in order to obtain
\begin{equation}\label{eq3.16}
C\|\nabla_\mathbf{a}u\|_{0,Q_{\tilde{\omega}}}\le \|u\|_{0,Q_\omega}+\|\partial_tu\|_{L^2((0,T),H^{-1}(\omega))}.
\end{equation} 
Inequalities \eqref{eq3.15} and \eqref{eq3.16} at hand, we can mimic the interpolation argument in the end of the proof of \cite[Proposition 3.1]{mach} to complete the proof.
\end{proof}

\subsection{Stabilization by an internal damping}

The following result was announced in the introduction. In this subsection we aim to prove it.

\begin{theorem}\label{theorem3.1}
There exists a constant $\varrho>0$, depending only on $\Omega$, $T$, $\Omega$ and $\Gamma_0$, so that
\[
\mathcal{E}_{u_0}^1(t)\le e^{-\varrho t}\mathcal{E}_{u_0}^1(0),\;\; u_0\in L^2(\Omega ).
\]
\end{theorem}

\begin{proof}
By density it is enough to give the proof when $u_0\in D(A_0)$. Fix then $u_0\in L^2(\Omega )$ and let $u(t)=e^{tA_1}u_0$. We decompose $u$ into two terms, $u=v+w$, 
with
\[
v(t)=e^{tA_0}u_0\;\; \mbox{and}\;\; w(t)=-i\int_0^te^{(t-s)A_0}cu(s)ds.
\]

As $\mathcal{E}_{u_0}^1$ is non increasing, we have
\[
\mathcal{E}_{u_0}^1(t)\le \mathcal{E}_{u_0}^1(0)=\frac{1}{2}\|u_0\|_{0,\Omega}^2.
\]
Hence
\[
\mathcal{E}_{u_0}^1(t)\le C\|v\|_{0,Q_\omega}^2  
\le C\|\sqrt{c}u\|_{0,Q_\omega}^2+C \|w\|_{0,Q_\omega}^2
\]
by Proposition \ref{proposition3.3}. Whence, using that $c\geq c_0>0$ a.e. in $\omega$,
\begin{equation}\label{eq3.17}
\mathcal{E}_{u_0}^1(t)\le C\|\sqrt{c}v\|_{0,Q}^2  \le C\|\sqrt{c}u\|_{0,Q}^2+C \|w\|_{0,Q}^2.
\end{equation}
On the other hand, it is straightforward to check that
\[
\|w\|_{0,Q}^2\le \| cu\|_{0,Q}^2\le \|c\|_\infty \|\sqrt{c}u\|_{0,Q}^2.
\]
This and \eqref{eq3.17} entail
\[
\mathcal{E}_{u_0}^1(t)\le C\|\sqrt{c}u\|_{0,Q}^2=-C\frac{d}{dt}\mathcal{E}_{u_0}^1(t).
\]
Or equivalently
\[
\frac{d}{dt}\mathcal{E}_{u_0}^1(t)\le -C^{-1}\mathcal{E}_{u_0}^1(t).
\]
This yields  the expected inequality in a straightforward manner.
\end{proof}

\subsection{Stabilization by a boundary damping}

In this subsection we take $d(x)=m(x)\cdot \nu (x)$, $x\in \Gamma_0$, which satisfies obviously the assumption required in Section 1.

Let $u_0\in V$ and recall the $\mathcal{E}_{u_0}^2(t)=\frac{1}{2}\left\|  \nabla_\mathbf{a} e^{tA_2}u_0 \right\|_{0,\Omega}^2$ satisfies
\[
\frac{d}{dt}\mathcal{E}_{u_0}^2(t)=-\|\sqrt{m\cdot\nu}\, u'(t)\|_{0,\Gamma_0}=-\|\sqrt{m\cdot\nu}\, \Delta_\mathbf{a}(t)\|_{0,\Gamma_0},\;\; t>0.
\]
Here $u(t)=e^{tA_2}u_0$.

Introduce,
\[
\mathscr{E}_{u_0}^2(t)=\Im (u(t)|m\cdot \nabla u(t))_{0,\Omega}.
\]

\begin{lemma}\label{lemma3.1}
For any $u_0\in V$ and $u(t)=e^{tA_2}u_0$, we have, where $t>0$,
\begin{equation}\label{3.1}
\frac{d}{dt}\mathscr{E}_{u_0}^2(t)=2\Re(\Delta_\mathbf{a}u|m\cdot \nabla u(t))_{0,\Omega}-n\|\nabla_\mathbf{a}u(t)\|_{0,\Omega}^2-\Re( (n+i)(m\cdot \nu)u(t)|u'(t))_{0,\Gamma_0}.
\end{equation}
\end{lemma}

\begin{proof}
By density it is sufficient to give the proof when $u_0\in D(A_2)$. In that case, we have
\[
\frac{d}{dt}\mathscr{E}_{u_0}^2(t)=\Im \left[ (u'(t)|m\cdot \nabla u(t))_{0,\Omega}+(u(t)|m\cdot \nabla u'(t))_{0,\Omega}\right].
\]

An integration by parts yields
\begin{align*}
(u(t)|m&\cdot \nabla u'(t))_{0,\Omega}= -(\mbox{div}(u(t)m)|u'(t))_{0,\Omega}+(u(t)(m\cdot \nu)|u'(t))_{0,\Gamma} 
\\
&=-n(u(t)|u'(t))_{0,\Omega}-(m\cdot \nabla u(t)|u'(t))_{0,\Omega}+(u(t)(m\cdot \nu )|u'(t))_{0,\Gamma} .
\end{align*}
Hence

\begin{align*}
\frac{d}{dt}\mathscr{E}_{u_0}^2(t)&=\Im \left[ (u'(t)|m\cdot \nabla u(t))_{0,\Omega}-n(u(t)|u'(t))_{0,\Omega}\right]
\\
&\qquad -\Im\left[(\nabla  u(t)\cdot m|u'(t))_{0,\Omega}+(u(t) (m\cdot \nu)|u'(t))_{0,\Gamma}\right].
\end{align*}
Since
\[
(u'(t)|m\cdot \nabla u(t))_{0,\Omega}-(\nabla  u(t)\cdot m|u'(t))_{0,\Omega}=2i\Im (u'(t)|m\cdot \nabla u(t))_{0,\Omega},
\]
we obtain
\[
\frac{d}{dt}\mathscr{E}_{u_0}^2(t)=2\Im (u'(t)|m\cdot \nabla u(t))_{0,\Omega}-n\Im(u(t),u'(t))_{0,\Omega}+\Im (u(t) (m\cdot \nu)|u'(t))_{0,\Gamma} .
\]

But $u'(t)=i\Delta_\mathbf{a}u(t)$. Therefore
\[
\frac{d}{dt}\mathscr{E}_{u_0}^2(t)=2\Re( \Delta_\mathbf{a}u|m\cdot \nabla u(t))_{0,\Omega}-n\Re( u(t),\Delta_\mathbf{a}(t))_{0,\Omega}+\Im (u(t) (m\cdot \nu) |u'(t))_{0,\Gamma} .
\]
This and
\[
(\Delta_\mathbf{a}u(t),u(t))_{0,\Omega}=-\|\nabla_\mathbf{a}u(t)\|_{0,\Omega}^2+(\partial_{\nu_\mathbf{a}}u(t)|u(t))_{0,\Gamma}
\]
entail
\begin{align}
\frac{d}{dt}\mathscr{E}_{u_0}^2(t)&=2\Re( \Delta_\mathbf{a}u|m\cdot \nabla u(t))_{0,\Omega}-n\|\nabla_\mathbf{a}u(t)\|_{0,\Omega}^2\label{3.2}
\\
&\qquad +n\Re(\partial_{\nu_\mathbf{a}}u(t)|u(t))_{0,\Gamma}+\Im (u(t) (m\cdot \nu) |u'(t))_{0,\Gamma} .\nonumber
\end{align}
Using that $\partial_{\nu_\mathbf{a}}u=-(m\cdot \nu)u'(t)$ on $\Gamma_0$ and $u=0$ on $\Gamma_1$, we get
\[
n\Re(\partial_{\nu_\mathbf{a}}u(t)|u(t))_{0,\Gamma}+\Im (u(t) (m\cdot \nu)|u'(t))_{0,\Gamma} =-\Re( (n+i)(m\cdot \nu)u(t)|u'(t))_{0,\Gamma_0}.
\]
In \eqref{3.2}, this identity yields
\begin{align*}
\frac{d}{dt}\mathscr{E}_{u_0}^2(t)=2\Re( \Delta_\mathbf{a}u(t)|m\cdot \nabla &u(t))_{0,\Omega}-n\|\nabla_\mathbf{a}u(t)\|_{0,\Omega}^2
\\
& -\Re ((n+i)(m\cdot \nu)u(t)|u'(t))_{0,\Gamma_0},
\end{align*}
which is the expected inequality.
\end{proof}

Henceforward, $\varkappa_1$ is the Poincar\'e constant of $V$.

\begin{lemma}\label{lemma3.2}
Assume that $\|\mathbf{a}\|_\infty \le \frac{1}{2\varkappa_1}$. Then, for any $u\in D(A_2)$, we have
\begin{align}
\Re(\Delta_\mathbf{a}u, m\cdot \nabla u)_{0,\Omega} \label{3.3}
&\le \frac{n-2}{2}(1+\delta(\|\mathbf{a}\|_\infty) )\|\nabla_\mathbf{a} u\|_{0,\Omega}^2
\\
&\quad +\Re( \partial_\nu u|m\cdot \nabla u)_{0,\Gamma_0}-\frac{1}{2}(|\nabla u|^2|m\cdot \nu)_{0,\Gamma_0} ,\nonumber
\end{align}
where the function $\delta$, depending only on $\Omega$ and $\Gamma_0$, satisfies $\delta (\rho)\rightarrow 0$ as $\rho \rightarrow 0$.
\end{lemma}

\begin{proof}
By simple integration by parts, we have
\begin{equation}\label{e3.1}
\Re(\nabla u|\nabla (m\cdot \nabla u))_{0,\Omega}=-\frac{n-2}{2}\|\nabla u\|_{0,\Omega}^2+\frac{1}{2}(|\nabla u|^2|m\cdot \nu)_{0,\Gamma} .
\end{equation}

But
\begin{align}
\Re(\Delta_\mathbf{a}u, m\cdot \nabla u)_{0,\Omega}&=\Re( \Delta u|m\cdot\nabla u)_{0,\Omega}- 2\Im (\mathbf{a}\cdot \nabla u|m\cdot \nabla u)_{0,\Omega}\label{e3.2}
\\
&\qquad +\Re( [i\mbox{div}(\mathbf{a})-|\mathbf{a}|^2]u|m\cdot \nabla u)_{0,\Omega}.\nonumber
\end{align}

Integrating by parts, the first term in the right hand side of inequality \eqref{e3.2} in order to get
\[
\Re(\Delta u|m\cdot\nabla u)_{0,\Omega}=-\Re(\nabla u|\nabla (m\cdot\nabla u))_{0,\Omega}+\Re( \partial_\nu u|m\cdot \nabla u)_{0,\Gamma} .
\]
This identity combined with \eqref{e3.1} yields
\begin{align}
\Re(\Delta u|m\cdot\nabla u)_{0,\Omega}&=\frac{n-2}{2}\|\nabla u\|_{0,\Omega}^2+\Re( \partial_\nu u|m\cdot \nabla u)_{0,\Gamma}-\frac{1}{2}( |\nabla u|^2|m\cdot \nu)_{0,\Gamma} .\label{e3.3}
\\&=\frac{n-2}{2}\left(\|\nabla_\mathbf{a} u\|_{0,\Omega}^2+2\Im (u|\mathbf{a}\cdot \nabla u)_{0,\Omega} -\||\mathbf{a}|u\|_{0,\Omega}^2\right)\nonumber
\\
&\qquad +\Re( \partial_\nu u|m\cdot \nabla u)_{0,\Gamma}-\frac{1}{2}( |\nabla u|^2|m\cdot \nu)_{0,\Gamma}  .\nonumber
\end{align}

Under the assumption on $\mathbf{a}$, straightforward computations show
\[
\|\nabla u\|_{0,\Omega} \le 2 \|\nabla_\mathbf{a} u\|_{0,\Omega}
\]
and
\[
\|u\|_{0,\Omega}\le 2\varkappa_1 \|\nabla_\mathbf{a} u\|_{0,\Omega}.
\]
These inequalities enable us to derive from \eqref{e3.3}
\begin{align}
\Re(\Delta u|m\cdot\nabla u)_{0,\Omega}\le \frac{n-2}{2}(1+\delta_0 )\|\nabla_\mathbf{a} u\|_{0,\Omega}^2&+\Re( \partial_\nu u|m\cdot \nabla u)_{0,\Gamma}\label{e3.4}
\\
&-\frac{1}{2}( |\nabla u|^2|m\cdot \nu)_{0,\Gamma} ,\nonumber
\end{align}
where
\[
\delta_0 =4(2\varkappa_1 +\varkappa_1^2)\|\mathbf{a} \|_\infty.
\]
Similarly, we have
\begin{align}
&\left| - 2\Im (\mathbf{a}\cdot \nabla u|m\cdot \nabla u)_{0,\Omega}+\Re([i\mbox{div}(\mathbf{a})-|\mathbf{a}|^2]u|m\cdot \nabla u)_{0,\Omega}\right| \label{e3.5}
\\
&\hskip 5cm \le \frac{n-2}{2}\delta_1\|\nabla_\mathbf{a} u\|_{0,\Omega}^2,\nonumber
\end{align}
the constant $\delta_1=\delta _1(\|\mathbf{a} \|_\infty)$ is so that $\delta_1(\rho) \rightarrow 0$ as $\rho \rightarrow 0$.

In light of \eqref{e3.4} and \eqref{e3.5}, we get from \eqref{e3.2}
\begin{align}
\Re(\Delta_\mathbf{a}u, m\cdot \nabla u)_{0,\Omega}\le \frac{n-2}{2}(1+\delta )\|\nabla_\mathbf{a} u\|_{0,\Omega}^2&+\Re( \partial_\nu u|m\cdot \nabla u)_{0,\Gamma}\label{e3.6}
\\
&-\frac{1}{2}( |\nabla u|^2|m\cdot \nu)_{0,\Gamma} .\nonumber
\end{align}
Here $\delta =\delta_0+\delta_1$.

On the other hand,
\begin{align}
\Re( \partial_\nu u|m\cdot \nabla u)_{0,\Gamma_1}&-\frac{1}{2}(|\nabla u|^2|m\cdot \nu)_{0,\Gamma_1}  \label{e3.7}
\\
&= ( |\partial_\nu u|^2|m\cdot \nu)_{0,\Gamma_1}-\frac{1}{2}(|\partial_\nu u|^2|m\cdot \nu)_{0,\Gamma_1}\nonumber
\\
&=\frac{1}{2}(|\partial_\nu u|^2|m\cdot \nu)_{0,\Gamma_1} \le 0\nonumber.
\end{align}

A combination of \eqref{e3.6} and \eqref{e3.7} yields
\begin{align*}
\Re(\Delta_\mathbf{a}u, m\cdot \nabla u)_{0,\Omega}\le \frac{n-2}{2}(1&+\delta )\|\nabla_\mathbf{a} u\|_{0,\Omega}^2
\\
&+\Re( \partial_\nu u|m\cdot \nabla u)_{0,\Gamma_0}-\frac{1}{2}( |\nabla u|^2|m\cdot \nu)_{0,\Gamma_0} .
\end{align*}
The proof is then complete.
\end{proof}

We are now able to prove the second exponential stabilization result. We recall that this result is the following
\begin{theorem}\label{theorem3.2}
There exists $0<\varsigma \le \frac{1}{2\varkappa_1}$, depending on $x_0$ and $\Omega$, with the property that, if $\|\mathbf{a}\|_\infty \le \varsigma$ and $\mathbf{a}=0$ on $\Gamma_0$, then there exists two constants $C>0$ and $\varrho>0$, depending only on $x_0$ and $\Omega$, so that
\[
\mathcal{E}_{u_0}^2(t)\le Ce^{-\varrho t}\mathcal{E}_{u_0}^2(0),\;\; u_0\in V.
\]
\end{theorem}

\begin{proof}
Let $u_0\in V$ and set $u(t)=e^{tA_2}u_0$. Since  $\mathbf{a}=0$ on $\Gamma_0$, we have 
\[
\Re( \partial_\nu u(t)|m\cdot \nabla u(t))_{0,\Gamma_0}=-\Re( (m\cdot \nu)u'(t)|m\cdot \nabla u(t))_{0,\Gamma_0}.
\]
This inequality and \eqref{3.3} entail
\begin{align*}
&2\Re(\Delta_\mathbf{a}u, m\cdot \nabla u)_{0,\Omega}\le (n-2)(1+\delta(\|\mathbf{a}\|_\infty) )\|\nabla_\mathbf{a} u\|_{0,\Omega}^2
\\
& \hskip 3cm- 2\Re( (m\cdot \nu )u'(t)|m\cdot \nabla u(t))_{0,\Gamma_0}-\Re( m\cdot \nu ||\nabla u|^2)_{0,\Gamma_0} .
\end{align*}
Using this inequality in \eqref{3.1}, we get
\begin{align}
\frac{d}{dt}\mathscr{E}_{u_0}^2(t)&\le - \|\nabla_\mathbf{a} u(t)\|_{0,\Omega}^2- 2\Re( (m\cdot \nu )u'(t)|m\cdot \nabla u(t))_{0,\Gamma_0}\label{3.11}
\\
&\hskip 2cm -\Re( m\cdot \nu ||\nabla u|^2)_{0,\Gamma_0} -\Re( (n+i)u(t)|u'(t))_{0,\Gamma_0} ,\nonumber
\end{align}
provided that $\delta\le \frac{1}{n-2}$. This last condition is satisfied whenever $\|\mathbf{a}\|_\infty \le \varsigma$, for some $0<\varsigma \le \frac{1}{2\varkappa_1}$.

Define, for $\epsilon>0$,
\[
\mathcal{E}_{u_0}^{2,\epsilon} =\mathcal{E}_{u_0}^2+\epsilon \mathscr{E}_{u_0}^2.
\]
From inequality \eqref{3.11}, we have
\begin{align}
&\frac{d}{dt}\mathcal{E}_{u_0}^{2,\epsilon} (t)\le -\epsilon \mathcal{E}_{u_0}^2(t)- (m\cdot \nu ||u'(t)|^2)_{0,\Gamma_0}\label{3.12}
\\
&-\epsilon \left[ 2\Re( (m\cdot \nu )u'(t)|m\cdot \nabla u(t))_{0,\Gamma_0}+\Re( m\cdot \nu ||\nabla u|^2)_{0,\Gamma_0} +\Re( (n+i)u(t)|u'(t))_{0,\Gamma_0}\right].\nonumber
\end{align}
Let $\|\mathbf{tr}\|$ be the norm of the trace operator
\[
u\in V \rightarrow \sqrt{m\cdot \nu}\, u_{|\Gamma_0} \in L^2(\Gamma_0),
\]
when $V$ is endowed with the norm $\|\nabla_\mathbf{a}\cdot \|_{0,\Omega}$. Then
\begin{align}
\left|((m\cdot \nu )(n+i)u(t)|u'(t))_{0,\Gamma_0}\right|&\le \frac{\|\mathbf{tr}\|^2}{2}(n^2+1)\|\sqrt{m\cdot \nu}u'(t)\|_{0,\Gamma_0}^2\label{3.13}
\\
&\hskip 2cm+\frac{1}{2\|\mathbf{tr}\|^2}\|\sqrt{m\cdot \nu}u(t)\|_{0,\Gamma_0}^2 \nonumber
\\
&\le \frac{\|\mathbf{tr}\|}{2}(n^2+1)\|\sqrt{m\cdot \nu}u'(t)\|_{0,\Gamma_0}^2+\frac{1}{2}\|\nabla_\mathbf{a}u(t)\|_{0,\Omega}^2.\nonumber
\end{align}
Also,
\begin{equation}\label{3.14}
2\left|u'(t)(m\cdot \nabla \overline{u(t)})\right|\le \|m\|_\infty^2|u'(t)|^2+ |\nabla u(t)|^2.
\end{equation}
If $\vartheta =\frac{\|\mathbf{tr}\|}{2}(n^2+1)+\|m\|_\infty^2$, then inequalities \eqref{3.13} and \eqref{3.14} in \eqref{3.12} entail
\[
\frac{d}{dt}\mathcal{E}_{u_0}^{2,\epsilon} (t)\le -\frac{\epsilon}{2} \mathcal{E}_{u_0}^2(t)- (1-\epsilon \vartheta)\|\sqrt{m\cdot \nu} u'(t)\|_{0,\Gamma_0}^2.
\]
That is
\begin{equation}\label{3.15}
\frac{d}{dt}\mathcal{E}_{u_0}^{2,\epsilon} (t)\le -\frac{\epsilon}{2} \mathcal{E}_{u_0}^2(t)\;\; \mbox{if}\;\; 1-\epsilon \vartheta \ge 0.
\end{equation}
On the other hand, as
\[
\mathscr{E}_{u_0}^2(t)\le 2\varkappa_1 \|m\|_\infty \|\nabla_\mathbf{a} u(t)\|_{0,\Omega}^2=2\varkappa \|m\|_\infty \mathcal{E}_{u_0}^2(t),
\]
we have
\[
\mathcal{E}_{u_0}^{2,\epsilon} (t)\le (1+2\varkappa_1 \|m\|_\infty)\mathcal{E}_{u_0}^2(t).
\]
This in \eqref{3.15} yields
\[
\frac{d}{dt}\mathcal{E}_{u_0}^{2,\epsilon} (t)\le -\epsilon \mu \mathcal{E}_{u_0}^{2,\epsilon} (t),\;\; 0<\epsilon \le \epsilon_0=\frac{1}{\vartheta}.
\]
Here $\mu =\frac{1}{2+4\varkappa_1 \|m\|_\infty}$. Hence
\[
\mathcal{E}_{u_0}^{2,\epsilon} (t)\le e^{-\epsilon \mu t}\mathcal{E}_{u_0}^{2,\epsilon} (0)
\]
But
\[
(1-2\varkappa_1 \|m\|_\infty)\mathcal{E}_{u_0}^2(t)\le \mathcal{E}_{u_0}^{2,\epsilon} (t)\le (1+2\varkappa_1 \|m\|_\infty)\mathcal{E}_{u_0}^2(t).
\]
Therefore
\[
\mathcal{E}_{u_0}^2 (t)\le \frac{1+2\varkappa_1 \|m\|_\infty}{2}e^{-\epsilon \mu t}\mathcal{E}_{u_0}^2 (0),\;\; 0<\epsilon \le \min \left(\epsilon_0, \frac{1}{4\varkappa_1 \|m\|_\infty}\right).
\]
The proof is then complete.
\end{proof}

\section{Additional comments}

\subsection{Exponential stabilization via a Carleman inequality}
\setcounter{equation}{0}

In this subsection we show that we can retrieve the exponential stabilization result in Theorem \ref{theorem3.1} by using an argument based on a Carleman inequality.

Assume that $\omega$ can be chosen in such a way that there exists $\psi \in C^4(\overline{\Omega})$ satisfying
\[
\psi >0\; \mbox{in}\; \overline{\Omega},\quad \nabla \psi \neq 0\; \mbox{in}\; \overline{\Omega\setminus \omega},\quad \partial _\nu \psi \leq 0\; \mbox{on}\; \Gamma  
\]
and the following  pseudo-convexity condition: there exists $\varpi>0$ so that
\[
|\nabla \psi (x)\cdot \xi |^2+\nabla ^2\psi(x)\xi \cdot \overline{\xi} \ge \varpi |\xi|^2,\;\; x\in \overline{\Omega\setminus \omega},\; \xi \in \mathbb{C}^n.
\]
Here where $\nabla ^2\psi=(\partial_{ij}\psi )$.

Note that since $\nabla ^2\psi$ is symmetric, $\nabla^2\psi \xi \cdot \overline{\xi}$ is real.

We call this condition on $\omega$ by $(\mathcal{G})$.

\smallskip
Let us provide a domain $\omega$ obeying to condition $(\mathcal{G})$. In fact, any neighborhood  $\omega$ of $\Gamma$ in $\Omega$ possesses this property. To see that, pick $\omega$ a neighborhood  of $\Gamma$ in $\Omega$, $x_0$ an arbitrary point in $\mathbb{R}^n\setminus\overline{\Omega}$ and $0\le \chi \in C_0^\infty (\Omega )$ satisfying $\chi=1$ in a neighborhood of $\overline{\Omega\setminus \omega}$. Then it is obvious to check that $\psi(x)=1+\chi(x)|x-x_0|^2$ satisfies all the conditions listed in $(\mathcal{G})$. This construction can be improved to include domains satisfying the condition for the exponential stabilization discussed in the multiplier method. To this end, fix again $x_0$ an arbitrary point in $\mathbb{R}^n\setminus\overline{\Omega}$ and set
\[
\Gamma_0=\{x\in \Gamma ;\; \nu(x)\cdot (x-x_0)>0\}.
\] 
Pick $\omega$ a neighborhood of $\Gamma_0$ in $\Omega$ so that $\overline{\omega}\cap \Gamma_1=\emptyset$ and let $0\le \chi \in C_0^\infty (\Omega )$ with $\chi =1$ in a neighborhood of $\overline{\Omega\setminus \omega}$ and $\textrm{supp}(\chi )\cap \overline{\Gamma_0}=\emptyset$. A straightforward computations show that $\psi(x)=1+\chi(x)|x-x_0|^2$ fulfills condition $(\mathcal{G})$. 

Substituting, if necessary, $\psi$ by $\psi +C$, for some large constant $C$, we can assume that 
\[
\psi >\frac{2}{3}\|\psi \|_\infty \;\; \mbox{in}\;\; \overline{\Omega}.
\]
In the sequel, the two functions $\theta$ and $\varphi$, defined on $Q$, are given by
\[
\theta (x,t)=\frac{e^{\lambda \psi (x)}}{t(T-t)},\quad \varphi (x,t)=\frac{e^{2\lambda \|\psi\|_\infty}-e^{\lambda \psi (x)}}{t(T-t)}.
\]
Here $\lambda$ is a parameter to be specified later.

Let 
\[
\mathscr{H}= \{ w\in L^2((0,T),H_0^1(\Omega ));\;  i\partial _t+\Delta _\mathbf{a}\in L^2(Q)\}.
\]

A straightforward modification of the proof \cite[Corollary 3.2]{MOR} gives

\begin{lemma}\label{lemma4.1}
There are three constants $\lambda _0\geq 1$, $s_0\geq 1$ and $C_0>0$ such that for all $\lambda \geq \lambda _0$, $s\geq s_0$ and $w\in \mathscr{H}$, it holds
\begin{align*}
&\|\sqrt{\lambda s\theta}e^{-s\varphi}  \nabla_\mathbf{a} w\|_{0,Q}+\|\lambda ^2s\theta\sqrt{s\theta}e^{-s\varphi}w\|_{0,Q}
\\
&\quad \leq C_0\left( \|e^{-s\varphi}( i\partial _t+\Delta _\mathbf{a})w\|_{0,Q}+\|\sqrt{\lambda s\theta}e^{-s\varphi}  \nabla_\mathbf{a} w\|_{0,Q_\omega}+\|\lambda ^2s\theta\sqrt{s\theta}e^{-s\varphi}w\|_{0,Q_\omega}\right).
\end{align*}
\end{lemma} 

Pick $u_0\in D(A_0)$ and let $u(t)=e^{tA_0}u_0$. Taking into account that
\[
\|u(t)\|_0=\|u_0\|_{0,\Omega}\;\; \mbox{and}\;\; \|\nabla_\mathbf{a}u(t)\|_{0,\Omega}=\|\nabla_\mathbf{a}u_0\|_{0,\Omega},
\]
we obtain by applying Lemma \ref{lemma4.1} the following observability inequality

\begin{corollary}\label{corollary4.1}
There exists a constant $C>0$, depending on $\Omega$, $\mathbf{a}$, $\omega$ and $T$, so that for any $u_0\in D(A_0)$ and $u(t)=e^{tA_0}u_0$, we have
\begin{equation}\label{4.1}
\|\nabla_\mathbf{a}u_0\|_{0,\Omega}\le C\left(\|\nabla_\mathbf{a}u\|_{0,Q_\omega}+\|u\|_{0,Q_\omega}\right).
\end{equation}
\end{corollary}

This observability inequality at hand, we can proceed similarly to the proof of Theorem \ref{theorem3.1} in order to get the following theorem.

\begin{theorem}\label{theorem4.1}
Let $\omega$ be a neighborhood of $\omega_0$ in $\Omega$, where $\omega_0$ obeys to the condition $(\mathcal{G})$. Then there exists a constant $\varrho>0$, depending only on $\Omega$, $T$, $\omega$, so that
\[
\mathcal{E}_{u_0}^1(t)\le e^{-\varrho t}\mathcal{E}_{u_0}^1(0),\;\; u_0\in L^2(\Omega ).
\]
\end{theorem}

Clearly, from the previous discussion, Theorem \ref{theorem4.1} improve Theorem \ref{theorem3.1}. However we do not know whether we can construct a domain $\omega$ obeying to condition $(\mathcal{G})$ but doesn't satisfy the assumption in Theorem \ref{theorem3.1}.

\begin{remark}\label{remark4.1}
{\rm
As in Theorem \ref{theorem3.1}, one step in the proof consists in establishing the following observability inequality
\[
\|u_0\|_{0,\Omega}\le C\|e^{tA_0}u_0\|_{0,Q_\omega},\;\; u_0\in L^2(\Omega ).
\]
According to \cite[Theorem 5.1]{Miller:2005},  
under the assumption of Theorem \ref{theorem4.1},  this inequality is equivalent to the following the so-called observability resolvent estimate: there exists two constants $\aleph_0$ and $\aleph_1$, depending on $\Omega$, $\omega$ and $\mathbf{a}$ so that, for any $\mu \in \mathbb{R}$ and $u\in D(A_0)$, we have
\[
\|u\|_{0,\Omega}^2\le \aleph_0\|(A_0-i\mu)u\|_0^2+\aleph_1\|u\|_{0,\omega}^2.
\]
}
\end{remark}

\subsection{Observability inequality in a product space}

We consider the case in which $\Omega =\Omega_1\times \Omega_2$, with $\Omega_j$ a $C^\infty$ bounded domain of $\mathbb{R}^{n_j}$, $j=1,2$ and $n_1+n_2=n$. Assume that
\[
\mathbf{a}(x_1,x_2)=(\mathbf{a_1}(x_1),\mathbf{a}_2(x_2))\in \mathbb{R}^{n_1}\oplus\mathbb{R}^{n_2},\;\; (x_1,x_2)\in \Omega .
\]
where $\mathbf{a}_j$ satisfies the same assumptions as $\mathbf{a}$ when $\Omega$ is substituted by $\Omega_j$, $j=1,2$. Denote by $A_{0,j}$ the operator $A_0$ when $\Omega=\Omega_j$ and $\mathbf{a}$ is substituted by $\mathbf{a}_j$, $j=1,2$.

For $u_{0,j}\in L^2 (\Omega _j)$, $j=1,2$, it is not hard to check that
\begin{equation}\label{4.2}
e^{tA_0}(u_{0,1}\otimes u_{0,2})=e^{tA_{0,1}}u_{0,1}\otimes e^{tA_{0,2}}u_{0,2}.
\end{equation}

Let $\omega_1$ be an open subset of $\Omega_1$, $Q_{\omega_1}=\omega_1\times (0,T)$, $\omega =\omega_1\times \Omega_2$ and $Q_\omega=\omega\times (0,T)$.

Following a simple idea in \cite{burq2}, we have

\begin{theorem}\label{theorem4.2}
Assume that there exists a constant $C>0$ so that the following observability inequality holds
\begin{equation}\label{4.3}
\|u_{0,1}\|_{0,\Omega_1}^2\le C\|e^{tA_{0,1}}u_{0,1}\|_{0,Q_{\omega_1}},\;\; u_{0,1}\in L^2(\Omega _1).
\end{equation}
Then
\[
\|u_0\|_{0,\Omega}^2\le C\|e^{tA_0}u_0\|_{0,Q_\omega},\;\; u_0\in L^2(\Omega ).
\]
\end{theorem}

\begin{proof}
Denote by $(\phi_k)_{k\ge 1}$ an orthonormal basis consisting of eigenfunctions of $A_{0,2}$. For $u_0\in L^2(\Omega )$, we have
\[
u_0(x_1,x_2)=\sum_{k\ge 1}\psi_k(x_1)\phi_k(x_2).
\]
Here
\[
\psi_k(x_1)=(u(x_1,\cdot )|\phi_k)_{0,\Omega_2}\in L^2(\Omega_1),\;\; k\ge 1
\]

In light of \eqref{4.2} we have, where $(i\lambda_k)\subset i\mathbb{R}$ is the sequence of eigenvalues of $A_{0,2}$,
\[
e^{tA_0}u_0(x_1,x_2)=\sum_{k\ge 1}e^{i\lambda_kt}e^{tA_{0,1}}\psi_k(x_1)\phi_k(x_2).
\]
We get by applying Parseval's identity
\[
\|e^{tA_0}u_0\|_{0,\Omega}^2=\sum_{k\ge 1}\|e^{tA_{0,1}}\psi_k\|_{0,\Omega_1}^2=\sum_{k\ge 1}\|\psi_k\|_{0,\Omega_1}^2=\|u_0\|_{0,\Omega}^2
\]
and
\begin{equation}\label{4.4}
\|e^{tA_0}u_0\|_{0,\omega}^2=\sum_{k\ge 1}\|e^{tA_{0,1}}\psi_k\|_{0,\omega_1}^2.
\end{equation}
On other hand, apply  observability inequality \eqref{4.3} in order to obtain
\[
\|u_0\|_{0,\Omega}^2=\sum_{k\ge 1}\|\psi_k\|_{0,\Omega_1}^2\le C\sum_{k\ge 1}\|e^{tA_{0,1}}\psi_k\|_{0,Q_{\omega_1}}^2.
\]
This and \eqref{4.4} entail
\[
\|u_0\|_{0,\Omega}^2\le C\|e^{tA_0}u_0\|_{0,Q_{\omega}}^2.
\]
This is the expected inequality.
\end{proof}

\textbf{Acknowledgements} We would like to thank the referees for their valuable comments which enabled us to improve  the paper.


\end{document}